\documentclass[oneside,11pt]{article}
\usepackage[margin=2cm]{geometry}  
\usepackage{amsmath,amssymb,amsthm}
\usepackage[mathlines]{lineno}
\usepackage[dvipsnames]{xcolor}
\usepackage[pagebackref,colorlinks,citecolor=Plum,urlcolor=Periwinkle,linkcolor=DarkOrchid]{hyperref}
\usepackage{graphicx}
\usepackage{enumitem}
\usepackage{nicefrac}
\usepackage{fontawesome}
\newtheorem{theorem}{Theorem}[section]
\newtheorem{lemma}[theorem]{Lemma}

\newtheorem{proposition}[theorem]{Proposition}
\newtheorem{definition}[theorem]{Definition}

\newtheorem{remark}[theorem]{Remark}

\newcommand{\Vor}{\operatorname{Vor}}
\newcommand{\Ter}{\operatorname{Ter}}

\DeclareMathOperator*{\argmin}{arg\,min}

\usepackage{caption}

\numberwithin{equation}{section}

\title{Voronoi cells in random split trees}

\author{Alexander Drewitz\thanks{Department Mathematik/Informatik, Universit\"at zu K\"oln. Email: adrewitz@uni-koeln.de} \and Markus Heydenreich\thanks{Mathematisches Institut, Ludwig-Maximilians-Universität München. Email: m.heydenreich@lmu.de}\and C\'ecile Mailler\thanks{Department of Mathematical Sciences, University of Bath, UK. Email: c.mailler@bath.ac.uk.}}

\newcommand{\bs}{\ensuremath{\boldsymbol}}
\newcommand{\sss}{\ensuremath{\scriptscriptstyle}}

\newcommand{\parent}{\overset{\sss\leftarrow}}

\newcommand{\op}{{o_{\mathbb P}}}
\newcommand{\Op}{{\mathcal O_{\mathbb P}}}

\begin{document}
\maketitle 

\begin{abstract}
We study the sizes of the Voronoi cells of $k$ uniformly 
chosen vertices in a random split tree of size $n.$
We prove that, for $n$ large, the largest of these $k$ Voronoi cells contains most of the vertices, 
while the sizes of the remaining ones are essentially all of order $n\exp(-\mathrm{const}\sqrt{\log n})$.
This discrepancy persists if we modify the definition of the Voronoi cells 
by (a) introducing random edge lengths (with suitable moment assumptions), 
and (b) assigning different ``influence'' parameters (called ``speeds'' in the paper) 
to each of the $k$ vertices. Our findings are in contrast to corresponding results on random uniform trees 
and on the continuum random tree, 
where it is known that the vector of the relative sizes of the $k$ Voronoi cells 
is asymptotically uniformly distributed on the $(k-1)$-dimensional simplex.
\end{abstract}

\section{Introduction}
\paragraph{Voronoi cells.} Consider a large graph $\mathcal G$, from which we choose $k$ vertices uniformly at random, and denote them by $U_1,\dots,U_k$. The \emph{Voronoi cell} $\Vor(U_j)$ of $U_j$ consists of those vertices that are closer in graph distance to $U_j$ than to any of the other chosen vertices $\{U_i\colon i=1,\dots,k; i\neq j\},$ with an arbitrary rule to break ties. We are studying the vector of proportional sizes \[\left(\frac{|\Vor(U_1)|}n,\dots,\frac{|\Vor(U_k)|}n\right),\] in the limit as $n\to\infty$, where $n=|\Vor(U_1)|+\cdots+|\Vor(U_k)|$ denotes the total number of vertices. 

In recent work, Addario-Berry et al.\ \cite{Chapuy} investigated this question for the case that $\mathcal G$ is a uniform tree, and proved that the limiting vector is uniform on the $(k-1)$-dimensional simplex. Indeed, they showed much more, namely that this is even true in a  limiting sense on the Brownian continuum random tree (henceforth CRT), and thus for all graph models that converge to the Brownian CRT in the Gromov-Hausdorff-Prokhorov topology: rooted plane trees, rooted unembedded binary trees, stacked triangulations, and others. 
Guitter \cite{Guitter_2017} proved the same uniform limit for the case that $\mathcal G$ is a random planar map of genus 0 and $k=2$. 
Chapuy \cite{Chapu19} made the far-reaching conjecture 
that the uniform limit is true for \emph{all random embedded graphs of fixed genus}. 

\paragraph{Our first contribution.}
In this paper, we look at the distribution of the Voronoi cells of $k$ 
uniform nodes in a random {\it split tree}. 
Split trees are a family of rooted trees introduced by Devroye~\cite{devroye-split} and later extended by Janson~\cite{Janson} who allowed trees of unbounded degrees: 
this family includes classical random trees such as the binary search tree, the random recursive tree, the preferential attachment tree 
(also called {\sc port} for ``plane oriented recursive tree'').
In our first main result, we prove that the largest of 
the Voronoi cells of $k$ uniform nodes in an $n$-node split tree 
contains a proportion~1 of all nodes.
We are also able to prove that the second, third, \dots, $k$-th largest Voronoi cells 
each contains an order $n\exp(-\mathrm{const}\sqrt{\log n})$ of all vertices.
We show that this result also holds when edges of the tree are given random i.i.d.\ lengths (of finite variance, or heavy-tailed but with finite mean), and defining the Voronoi cells with respect to the distance induced by these edge lengths instead of the graph distance.

This result is in contrast with the findings of~\cite{Chapuy} for the uniform random tree equipped with the graph distance:
the distribution of the sizes of Voronoi cells is balanced in the case of the uniform random tree 
(and other trees whose scaling limit is the CRT), 
while we show a ``winner takes it all'' behaviour in the case of split trees.
This difference in behaviour should not be surprising: 
it is well-known that split trees have a very different shape from the uniform random trees 
(and other random trees whose scaling limit is the CRT): for example, 
the typical height of an $n$-node split tree is $\log n$, 
while the typical height of the uniform random tree is $\sqrt n$.
In that sense, split trees belong to another universality class of random trees 
(as opposed to trees whose scaling limit is the CRT), 
and our first main result corresponds to the findings of \cite{Chapuy} for this second universality class.
Similarly to~\cite{Chapuy} conjecturing that their result generalises to maps that 
scale to the random Brownian map, 
one might expect that the behaviour we prove for random split trees might also be exhibited by other graphs such as preferential attachment graphs 
and other scale-free models such as the configuration model. 
However, our proofs cannot be straightforwardly generalised.


\paragraph{Extension to a competition/epidemics model.}
The Voronoi cells can be seen as the result of a \emph{competition model} 
where $k$ agents are claiming territory with uniform speed until they reach 
vertices that are already claimed by another agent. 
This procedure stops when all vertices are claimed by some agent; 
the final territories are the same as the Voronoi cells. 
As discussed above, our first result is that -- unlike in the case of uniform trees -- 
the final territories are rather unbalanced: 
while one agent will claim almost the entire tree, 
the rest has to live on a rather small territory. 
(This behaviour also persists when we introduce random edge lengths.)

One can also see this competition model as a 
competition between~$k$ mutually exclusive epidemics, 
which are started at $k$ uniform vertices of a split tree, 
and which all spread at constant and equal speed.
Our second main result is that,
if the speed of transmission varies among the different epidemics, 
then the fastest epidemic spreads over order $n$ of the vertices.
We are also able to estimate precisely the number of 
nodes that get infected by each of the slower epidemics.

Note that this ``winner takes it all'' effect has already been observed in a competing first-passage percolation model on the configuration model with tail exponent $\tau\in(2,3)$ (see \cite{deijfen2016}). The main difference with our model is that epidemics spread deterministically in our model, and randomly at a given rate in the competing first-passage percolation model.
In both models, one epidemic occupies eventually almost all of the available territory. 
In the case of different speeds, this is the fastest one, but in the case of equal speed this is determined by the initial position (see \cite{BaronHofstKomja15,HofstKomja15}, where this is proved for the competing first-passage percolation model). The competing first-passage percolation model on random regular graphs exhibits similar behaviour~\cite{Antun17}. 
This suggests that the uniform limiting proportion of Voronoi cells is not true on complex networks. Instead, our results support the belief that for competing epidemics on networks with small distances (``small-world graphs'') there is one dominating epidemic.

\paragraph{Information on the typical shape of a random split tree.}
The asymptotic sizes of the Voronoi cells (or territories) 
of $k$ nodes chosen uniformly at random in a tree 
gives information on the typical shape of a tree.
In fact, to prove our main result, 
we prove two results that may be of independent interest because they give information of the typical shape of a random split tree: 
(1) in Proposition~\ref{prop:profile} we show convergence in 
probability of the ``profile'' of a random split tree, 
and (2) in Proposition \ref{prop:fringe_trees}, we prove asymptotic results for the 
size of a typical ``extended'' fringe tree in a random split tree.

(1) The profile of a random tree is the distribution of the height (distance to the root)
of a node taken uniformly at random in the tree. 
If the tree is random then its profile is a random measure.
In Proposition~\ref{prop:profile}, we show that the profile of a random split tree behaves asymptotically (in probability) as a Gaussian centred around $\mathrm{const}\log n$ and of standard deviation $\mathrm{const}\log n$. 
Our framework includes the cases of the random binary and $m$-ary search trees, the random recursive tree and the preferential attachment trees, for which convergence of the profile is already known in the almost sure sense (see \cite{CDJ01}, \cite{MM17}, and \cite{Katona05}, respectively).

(2) Fringe trees are subtrees that are rooted at 
an ancestor of a node taken uniformly at random in the tree (or at the uniform node itself).
Oftentimes, this ancestor is chosen to be at constant distance of the uniform node (see, e.g.~\cite{JansonHolm} and the references therein).
In Proposition \ref{prop:fringe_trees}, we extend this definition to allow the ancestor to be at distance to the uniform node that tends to infinity with $n$, the number of nodes in the whole split tree.

\bigskip
The main technical obstables in our proofs comes from the three levels of randomness: 
(a) the trees we consider are random split trees, 
(b) we then sample i.i.d.\ random edge-lengths, and 
(c) we finally sample $k$ nodes uniformly at random in the tree.
The advantage of our approach is that the framework we consider is very wide: 
the random split trees we consider include, among others, the binary and $m$-ary search trees, the random recursive tree, and the preferential attachment tree;
our edge-length distribution can be of finite variance, or heavy-tailed with finite mean;
we allow the different epidemics to have identical or different speeds.

In the rest of this section, we define our model (Section~\ref{sub:def_split}) and state our main results (Section~\ref{sub:results}).

\subsection{Trees and random split trees}\label{sub:def_split}
In this paper, we use the Ulam-Harris definition of $m$-ary trees:
let $m\in\mathbb N$ and 
\[\mathcal D_m = \{1, 2, \ldots, m\}^* = \{\varnothing, 1, 2, \ldots, m, 11, 12, \ldots 1m, \ldots\},\] 
 be the set of all finite words on the alphabet $\{1, 2, \ldots, m\}$.
 We further consider the case of infinitary trees, where $m=\infty$ and $\mathcal D_\infty = \mathbb N^*$. 
 We henceforth formulate our results for finite and infinite $m$ in a unified fashion (unless stated explicitly); finite tuples, such as in \eqref{eqSig} below, should be interpreted as infinite sequence whenever $m=\infty$. 
\begin{definition}\label{def:tree}
An $m$-ary tree is a subset $t$ of $\mathcal D_m$ such that for all $w=w_1\cdots w_\ell\in t$, all the prefixes of $w$ are in $t$, i.e.\ for all $i\in \{0, \ldots, \ell\}$ one has $w_1\cdots w_i\in t$.
\end{definition}
See Figure~\ref{fig:UH} for an example of a $3$-ary tree. 
In the following, we collect some standard vocabulary and notations; they reflect the fact that a tree is often seen as a genealogical structure: 
\begin{itemize}[noitemsep]
\item words are called ``nodes'';
\item the prefixes of a word are its ``ancestors'':
we write $v\prec w$ if $v$ is an ancestor of $w,$ and $v\preccurlyeq w$ if $v$ is $w$ or an ancestor of $w$;
\item the longest of the (strict) prefixes of a word $w$ is its ``parent'', which we denote by $\parent v$;
\item a node is a ``child'' of its parent, and it is a ``descendant'' of each of its ancestors;
\item the ``siblings'' of a node $v$ are all those nodes different from $v$ that share the same parent with $v;$ its ``left-siblings'' (resp.\ ``right-siblings'') are all its siblings that are smaller (resp.\ larger) in the lexicographic order;
\item the word $\varnothing$ is the ``root'' of the tree;
\item the ``height'' of a node is the number of letters in the word (the root is at height~0);
\item the ``last common ancestor'' of two nodes is the longest prefix shared by the two nodes:
we write $u\wedge v$ for the last common ancestor of nodes $u$ and $v$.
\end{itemize}

In particular, the definition of a tree can be immediately rephrased using this new vocabulary reflecting the genealogical point of view:
a tree is a set of nodes such that if a node is in the tree, 
then all its ancestors must also be in the tree.

\begin{figure}
\begin{center}
\includegraphics[width=5cm]{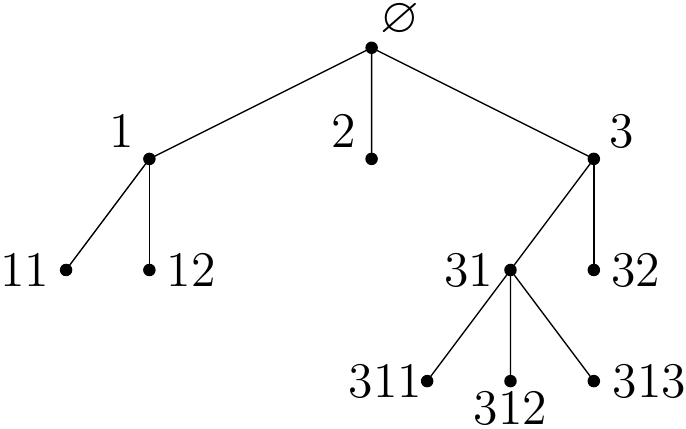}
\end{center}
\caption{The 3-ary tree $\{\varnothing, 1,2,3,11,12,31,32,311,312,313\}$.
Node 312 is the ``second child of the first child of the third child of the root'', its parent is node $31$, its siblings are $311$ and $313$. The last common ancestor of $32$ and $313$ is $3$.
}
\label{fig:UH}
\end{figure}


\medskip
We now define a probability distribution on the set of $m$-ary trees:
it is the distribution of ``split trees'' first introduced by Devroye~\cite{devroye-split}, 
but generalised to possibly infinite arity as in~\cite{Janson}.
Let $\nu$ be a probability distribution on the set
\begin{equation}\label{eqSig}
	\Sigma_m = \Big \{(v_1, \ldots, v_m) \in [0,1]^m \colon \sum_{i=1}^m v_i = 1 \Big \},
\end{equation}
and $(\bs Y(w))_{w\in\mathcal D_m}$ 
be a family of i.i.d.\ $\nu$-distributed random vectors.
For each node $w=w_1\cdots w_\ell\in \mathcal D_m$, 
we let $Z_w = Y_{w_\ell}(\parent w)$, 
where $\parent w$ is the parent of $w$, i.e.\ $\parent w = w_1\cdots w_{\ell-1}$ and with $Y_{w_\ell}(\parent w)$ denoting the $w_\ell$-th coordinate of the vector $\bs Y(\parent w)$ (see Figure~\ref{fig:split_tree} for an example: $\bs Y(3) = (.1,.4,.5)$ and thus $Z_{32}=.4$).

We also let $(X_n)_{n\geq 0}$ be a sequence of i.i.d.\ random variables 
uniformly distributed on $[0,1]$, and independent 
from the sequence $(\bs Y(w))_{w\in\mathcal D_m}$.

We need one last definition to define our sequence of random split trees: 
Given a tree $t$, we denote by $\partial t$ the nodes of $\mathcal D_m$ that are not in $t$ but whose parent is in $t$, and we call the elements of this set the ``leaves'' of $t$. It is not hard to see that if $t$ has $n$ nodes, then $\partial t$ has cardinality $(m-1)n+1$ (see Figure~\ref{fig:split_tree}). 

We can now define the sequence $(\tau_n)_{n\geq 1}$ of random trees
recursively as follows.
\begin{itemize}
\item the tree $\tau_1$ is defined to consist of the root only, i.e.\ $\tau_1 = \{\varnothing\}$.
\item for $n\geq 1$ arbitrary, given $\tau_n$, we define $\tau_{n+1}$ as the tree obtained by adding one node to $\tau_n$ as follows: 
\begin{itemize}
\item 
We subdivide the interval $[0,1]$ in subintervals indexed by $\partial\tau_n$ 
of respective lengths $\prod_{\varnothing\neq v\preccurlyeq w} Z_v$, 
for all $w\in\partial\tau_n$.
(Note that, by definition, $\sum_{w\in\partial\tau_n}\prod_{\varnothing\neq v\preccurlyeq w} Z_v = 1;$ see Figure~\ref{fig:split_tree} for an example, and observe that some points form part of several intervals.)
\item We set $\xi(n+1) = w$ if $X_{n+1}\in [0,1]$ belongs to the part indexed by $w$ of this partition of $[0,1]$, and finally set $\tau_{n+1} = \tau_n \cup\{\xi(n+1)\};$  note that this is well-defined almost surely.
\end{itemize}
\end{itemize}
The sequence of random trees $(\tau_n)_{n\geq 1}$ is called the {\em random split tree of split distribution~$\nu$} (which we recall is the distribution of the $\bs Y(w)$'s).

\begin{figure}
\begin{center}
\includegraphics[width=5cm]{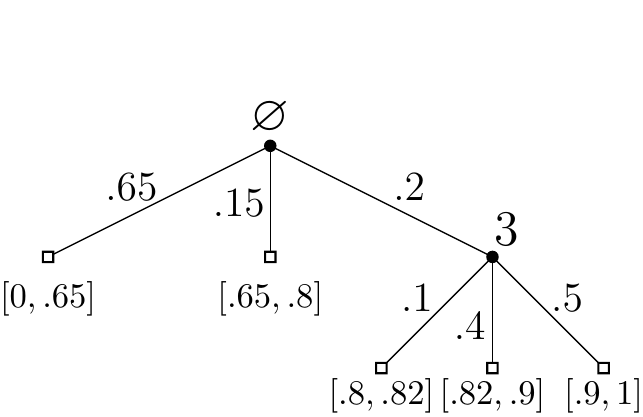}
\end{center}
\caption{A realisation of the 3-ary split tree $\tau_2$, here we have $\tau_2 = \{\varnothing,3\}$. 
The labels on the edges represent the values of $(\bs Y(w))_{w\in\tau_2}$: 
for example, $\bs Y(\varnothing) = (.65,.15,.2)$. 
The value of $Z_w$ is thus the label on the edge from~$w$ to its parent: 
for example, $Z_{31}=.1$. 
The nodes that are marked by a square are the elements of $\partial \tau_2$, 
underneath each leaf is written the corresponding part in the partition used to build $\tau_3$. 
For example, the part corresponding to $32$ is of length $Z_3Z_{32} =.2\times .4 = .08$.}
\label{fig:split_tree}
\end{figure}

This definition incorporates a variety of different random trees that are classical in the literature:
\begin{itemize}
\item If $m=2$ and $\nu$ is the distribution of $(Y, 1-Y)$, where $Y$ is uniform on $[0,1]$, then $(\tau_n)_{n\geq 1}$ is the {\it random binary search tree} (see~\cite[Table~1]{devroye-split}).
\item If $\nu$ is the uniform distribution on the simplex $\Sigma_m$ for $m$ finite, then $(\tau_n)_{n\geq 1}$ is the {\it random $m$-ary search tree} (see~\cite[Table~1]{devroye-split}).
\item If $m=\infty$ and $\nu$ is $\mathrm{GEM}(0,1)$ on $\Sigma_{\infty}$, then $(\tau_n)_{n\geq 1}$ is the {\it random recursive tree} (see~\cite[Cor.~1.2]{Janson}).
\item If $m=\infty$ and $\nu$ is $\mathrm{GEM}(\nicefrac12,\nicefrac12)$, then $(\tau_n)_{n\geq 1}$ is the {\it random preferential attachment tree} (see~\cite[Cor.~1.3]{Janson}).
\end{itemize}

\begin{remark}
For $\alpha\in [0,1]$ and $\theta>0$, the Griffiths-Engen-McCloskey distribution
$\mathrm{GEM}(\alpha, \theta)$ is defined as 
the distribution of the sequence $(A_n)_{n\geq 1}$ defined as follows:
sample $(B_i)_{i\geq 1}$ a sequence of independent random variables 
of respective distributions~$\mathrm{Beta}(1-\alpha, \theta+i\alpha)$,
and, for all $n\geq 1$, set $A_n = B_n\prod_{i=1}^{n-1}(1-B_i)$.
\end{remark}

\subsection{Voronoi cells and final territories}\label{sub:results}
In this paper, our aim is to investigate the sizes of the Voronoi cells corresponding to~$k$ nodes taken uniformly at random in the $n$-node random split tree $\tau_n$ defined in Section~\ref{sub:def_split}. In this context, we will also accommodate for the setting of  having random edge lengths between the nodes: 
let $\varpi$ be a probability distribution on $(0,\infty)$ and
let $(L_w)_{w\in\mathcal D_m}$ 
be a sequence of i.i.d.\ random variables of distribution~$\varpi$,
and we define the distance between two nodes as the sum of the length of the edges on the unique shortest path between them; see Figure~\ref{fig:distance} for an example.
\begin{definition}\label{def:Distances}
For all families $\bs \ell = (\ell_w)_{w\in\mathcal D_m}$ of positive random variables we define a distance $d_{\bs \ell}$ on $\mathcal D_m$ as follows: 
for all pairs of nodes $u$ and $v$ in $\mathcal D_m$ (for all $m\geq 2$),
let
\[d_{\bs \ell}(u,v) 
:= \sum_{u\wedge v\prec w\preccurlyeq u} \ell_w
+  \sum_{u\wedge v\prec w\preccurlyeq v} \ell_w.\]
For all nodes $w\in\mathcal D_m$, we denote by $|w|_{\bs \ell}:=d_{\bs\ell}(\varnothing, w)$.
\end{definition}
This definition holds for any fixed sequence $\bs \ell$ of edge lengths: all along the paper, we use the distance $d_{\bs L}$, where $\bs L = (L_w)_{w\in\mathcal D_m}$ is the sequence of i.i.d.\ random edge lengths.
\begin{figure}
\begin{center}
\includegraphics[width=4.5cm]{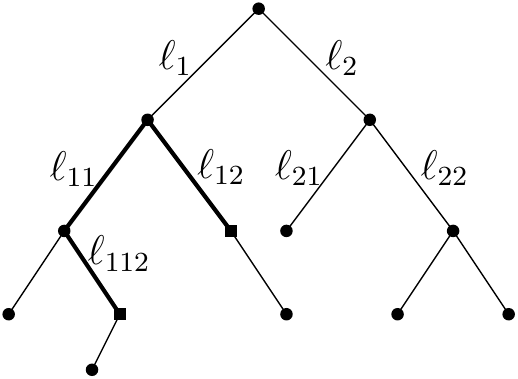}
\end{center}
\caption{A binary tree. The distance between the nodes $112$ and $12$ (marked as squares on the picture) with respect to the sequence $\bs \ell$ is $\ell_{11}+\ell_{112}+\ell_{12}$ (the sum of the length of the bold edges) because their last common ancestor is $1$.}
\label{fig:distance}
\end{figure}
Also, note that if $\ell_w=1$ for all $w\in\mathcal D_m$, then $d_{\bs \ell}$ corresponds to the graph distance in the graph whose nodes are all elements of $\mathcal D_m,$ and where there is an edge between two nodes if and only if one is the parent of the other.

\begin{definition}\label{def:Vor}
Let $u_1, \ldots, u_k$ be $k$ nodes in an $m$-ary tree $t$, and $d$ a distance on $\mathcal D_m$.
We define the Voronoi cells of $u_1, \ldots, u_k$ as follows: for all $1\leq i\leq k$,
\begin{align*}
\Vor_{t,d}^{i}(u_1, \ldots, u_k) 
&=\big\{w\in t \colon d(w,u_i)\le d(w,u_j)\text{ for }j=1,\dots,i-1\text{ and } \\ &\qquad d(w,u_i)<d(w,u_j)\text{ for }j=i+1,\dots,n \big \}.
\end{align*}
We say that $\Vor_{t,d}^i(u_1, \ldots, u_k)$ is the Voronoi cell of $u_i$ 
(with respect to $u_1, \ldots, u_k$).
\end{definition}

\begin{remark}
The idea of Definition~\ref{def:Vor} is that $\Vor_{t,d}^i(u_1, \ldots, u_k)$ contains all the nodes that are closer to $u_i$ than to any of the other $u_j$'s for distance $d$ on $t$.
The difference between `$<$' and `$\le$' induces a simple rule to break ties (in case of equal distances, the vertex with smaller index is preferred). 
However, since the number of boundary vertices is of constant order, the choice we make about how to break ties has no impact on our results.
\end{remark}


A possible interpretation of Voronoi cells is in terms of epidemics: imagine that $k$ competing epidemics start spreading at speed one from, respectively, $u_1, \ldots, u_k$, 
and that once a node is infected by an epidemic, then it becomes immune to all others. 
If two or more epidemics reach one node at the same time, then the node gets infected with the epidemics that started at the $u_i$ with smallest index.
In this context, the Voronoi cells are the final territories of the $k$ infections, that is, the Voronoi cell of $u_i$ contains all the nodes that got infected by the epidemic that started at node $u_i$. From this point of view, it is natural to consider the case when the epidemics spread at different speeds:

\begin{definition}
Let $t$ be an $m$-ary tree and denote by $d$ be a distance on $\mathcal D_m.$ Furthermore, let $u_1, \ldots, u_k$ be nodes in $t$ and let $s_1, \ldots, s_k\in (0,+\infty),$ the \lq speeds of the epidemics\rq.
We define the final territories of $(u_1,s_1), \ldots, (u_k, s_k)$ as
\begin{align*}
\Ter_{t,d}^i((u_1, s_1), \ldots, (u_k, s_k))
&=\Big\{w\in t \colon d(w,u_i) \le \frac{s_i+s_j}{s_i}\,d(w,u_j) \text{ for }j=1,\dots,i-1\text{ and } \\
& \qquad(w,u_i)< \frac{s_i+s_j}{s_i}\,d(w,u_j)
\text{ for }j=i+1,\dots,n\} \Big\}.
\end{align*}
\end{definition}

\subsubsection{Main results}
Our main result provides asymptotic statements on the sizes of $k$ epidemics in the case when the epidemics have different speeds. We first state the result in the simpler case when all epidemics have the same speed (Theorem \ref{th:voronoi}) and then extend it to the setting where different speeds are admissible (Theorem \ref{th:territories}). Both theorems apply to finite ($m\in\{2, 3, \ldots\}$) as well as infinite ($m=\infty$) arity. 
They hold under the following hypothesis on the split-vector distribution $\nu$ and the edge length distribution $\varpi$:
\begin{itemize}
\item[(A1)] 
(i) If $\mathrm{Supp}(\nu)$ denotes the support of the probability distribution $\nu$, and $\mathbf e_i$ is the $m$-dimensional vector whose coordinates are all equal to~0 except the $i$-th coordinate which is equal to one, then
\[\mathrm{Supp}(\nu)\setminus \{\mathbf e_1, \ldots, \mathbf e_m\}\neq \varnothing.
\]
(ii) Moreover, if $(Y_1, \ldots, Y_m)\sim \nu$, $U$ is a uniform random variable on $[0,1]$,
and\footnote{By convention, we set $\sum_{i=1}^0 a_i = 0$ for each sequence $(a_i)_{i\geq 0}$ of real numbers.}
\begin{equation}\label{eq:barY}
\bar Y = \sum_{i=1}^m Y_i\bs 1\Big\{\sum_{j=1}^{i-1} Y_j\leq U<\sum_{j=1}^{i}Y_j\Big\}
\end{equation}
is the size-biased version of the marginals of $\nu$,
then $\mu:=\mathbb E[\log \nicefrac{1}{\bar Y}]> 0$
and $\sigma^2:=\mathrm{Var}(\log \bar Y)<+\infty.$
\item[(A2)] If $L\sim\varpi$ then either $\mathrm{Var}(L) \in [0,+\infty)$, in which case we set $\alpha:=2$, or there exists $\alpha\in (1,2)$ and a function $\ell$ slowly varying at infinity, such that $\mathbb P(L\geq x) = x^{-\alpha}\ell(x)$. In particular, $\mathbb E L < \infty$ in this case.
\end{itemize}

Assumption (A1-i) just excludes the trivial case when the $n$-node split tree is almost surely equal to a line of $n$ nodes hanging under each other under the root. Assumptions (A1-ii) and (A2) give some control over the moments of respectively the split vectors and the edge lengths: these assumptions will be used when applying laws of large numbers and of the iterated logarithm, as well as central limit theorems to sum of independent copies of these random variables.

\begin{theorem}\label{th:voronoi}
Let $\varpi$ be a probability distribution on $(0,+\infty)$, and $\nu$ be a probability distribution on $\Sigma_m$. Let $(\tau_n)_{n\geq 1}$ be the random split tree of split distribution $\nu$, and $\bs L = (L_w)_{w\in\mathcal D_m}$ be a sequence of i.i.d.\ random variables 
of distribution $\varpi$, independent of $(\tau_n)_{n\geq 1}$.


For each $n\geq 1$, let $U_1(n), \ldots, U_k(n)$ be $k$ nodes taken uniformly at random among the $n$ nodes of $\tau_n$; we let $V_{\sss (1)}(n)\geq \ldots \geq V_{\sss (k)}(n)$ be the sizes of their Voronoi cells in $\tau_n$ with respect to the distance $d_{\bs L}$, ordered in decreasing order.

Under Assumptions {\rm (A1)} and {\rm (A2)}, we have in distribution when $n\to+\infty$,
\begin{equation} \label{eq:VorConvDist}
\frac1{(\log n)^{\nicefrac1\alpha}}\big(\log (V_{\sss (2)}(n)/n), \ldots, \log (V_{\sss (k)}(n)/n)\big) \Rightarrow
\frac{\frak v}{2\mathbb EL}
\big(\Psi_{(1)}-\Psi_{(2)}, \ldots, \Psi_{(1)}-\Psi_{(k)}\big);
\end{equation}
here, $\Psi_{(1)}\leq \cdots\leq \Psi_{(k)}$ is the order statistics of $k$ i.i.d.\ random variables whose distribution is 
$\mathcal N(0, \mathrm{Var}(L) + \sigma^2 (\mathbb EL)^2)$  if $\mathrm{Var}(L)<+\infty$, and an $\alpha$-stable distribution otherwise, and where
\begin{equation}\label{eq:def_frak_v}
\frak v = \begin{cases}
\mu^{-\nicefrac12} & \text{ if }\mathrm{Var}(L)<+\infty\\
\mu^{1-\nicefrac1\alpha} &\text{ otherwise.}
\end{cases}
\end{equation}
\end{theorem}
In words, the above amounts to the fact that the second, third, \dots, $k$th largest component each occupies a proportion of roughly $\exp\{-\Psi(\log n)^{1/\alpha}\}$ of the vertices, where $\Psi$ is some explicit positive random variable. This implies that asymptotically and in distribution, the entire mass is allocated to the largest component (which, by construction, belongs to the vertex closest to the root). The allocation for split trees is therefore qualitatively very different from the allocation in the universality class of the continuum random tree, where the limit of the proportions of the masses is known to be uniform \cite{Chapuy}.

We now extend the results of the previous theorem to the case of different speeds at which the uniformly chosen vertices claim territory (use the same notation as in Theorem~\ref{th:voronoi}). 

\begin{theorem}\label{th:territories}
For all $s_1, \ldots, s_k\in (0,+\infty)$, we let $W_{(1)}(n)\geq  \ldots \geq W_{(k)}(n)$ be the sizes (ordered in decreasing order) of the final territories in $\tau_n$, equipped with the distance $d_{\bs L}$, of $k$ epidemics of respective speeds $s_1, \ldots, s_k$ and starting from $U_1(n), \ldots, U_k(n),$ respectively. 
Without loss of generality, we assume that $s_1 = \ldots = s_j > s_{j+1} \geq \ldots\geq s_k$ for some $j\in\{1,\dots,k\}$. 

Then, under Assumptions {\rm (A1)} and {\rm (A2)}, we have in distribution when $n\to+\infty$,
\begin{equation} \label{eq:terPrincipalConv}
\frac1{(\log n)^{\nicefrac1\alpha}}\big(\log (W_{(2)}(n)/n), \ldots, \log (W_{(j)}(n)/n)\big) \Rightarrow
\frac{\frak v}{2\mathbb EL}
\big(\Psi_{(1)}-\Psi_{(2)}, \ldots, \Psi_{(1)}-\Psi_{(j)}\big),
\end{equation}
where $\frak v$ is defined in~\eqref{eq:def_frak_v},
$\Psi_{(1)}\leq \cdots\leq \Psi_{(j)}$ is the order statistics of $j$ i.i.d.\ random variables whose distribution is 
$\mathcal N(0, \mathrm{Var}(L) + \sigma^2 (\mathbb EL)^2)$  if $\mathrm{Var}(L)<+\infty$, and an $\alpha$-stable distribution otherwise.

Furthermore, if $\mathbb EL^2<+\infty$, then, for all $i \in \{j+1, \ldots, k\}$,
\[\frac{\log(W_{(i)}(n)/n) + \frac{s_1-s_i}{s_1+s_i}\log n}{\sqrt{\frac{s_1-s_i}{s_1+s_i}\log n}}
\Rightarrow \mathcal N\bigg(0, \frac{\mathrm{Var}(\log \bar Y)\mathbb EL}{\mathbb E[\log\bar Y]^2}+\mathrm{Var}(L)\bigg).\]
Otherwise, if $\mathbb P(L\geq x) = x^{-\alpha}\ell(x)$ for some function $\ell$ slowly varying at infinity and some $\alpha\in(1,2)$, then, for all $i \in \{j+1, \ldots, k\}$,
\[\frac{\log (W_{(i)}(n)/n)+\frac{s_1-s_i}{s_1+s_i} \log n}{(\frac{s_1-s_i}{s_1+s_i} \log n)^{\nicefrac1\alpha}}
\Rightarrow \frac{\mu^{1-\nicefrac1\alpha}}{\mathbb E L}\Upsilon(\alpha),
\]
where $\Upsilon(\alpha)$ is an $\alpha$-stable distribution.
In particular, in both cases, we have
\begin{equation} \label{eq:terMinorConv}
\frac{\log(W_{(i)}(n)/n)}{\log n} \to \frac{s_i-s_1}{s_i+s_1} 
\end{equation}
in probability when $n\to+\infty$,
for each $i \in \{j+1, \ldots, k\}$.
\end{theorem}

Note that, given that the slower epidemics all have very small territories (cf.~\eqref{eq:terMinorConv}), 
the $j$ fastest territories behave as in~Theorem~\ref{th:voronoi}, which -- at least heuristically -- entails~\eqref{eq:terPrincipalConv}.

It is also interesting to note that in their first asymptotic order given by~\eqref{eq:terMinorConv}, 
the sizes of the slow epidemics do not depend on the edge length. 
An intuitive indication towards this fact is that replacing $L$ by $cL$ for a positive constant $c$ does not change the sizes of the territories.  
In a similar vein, the right-hand sides of~\eqref{eq:VorConvDist} and \eqref{eq:terPrincipalConv} also remain unchanged upon replacing $L$ by $cL$, as expected.

As a by-product of our proof of Theorems~\ref{th:voronoi} and~\ref{th:territories}, we get the following result on the convergence of the profile of random split trees,
which, as far as we are aware, is a new result in the context of split trees:
\begin{proposition}\label{prop:profile}
Let $(\tau_n)_{n\geq 1}$ be the random split tree of split distribution $\nu$, and let, for all integer $n$, $\pi_n = \frac1n\sum_{i=1}^n \delta_{|\nu_i|}$ be the random profile of $\tau_n$, 
where we recall that $|\nu_i|$ is the height of the node inserted at time $i$ in $(\tau_n)_{n\geq 1}$.
If $\nu$ satisfies Assumption {\rm (A1)}, then
\begin{equation}\label{eq:cv_profile}
\pi_n\big(\,\cdot\,\sqrt{(\log n)/\mu^3}+ (\log n)/\mu\big) \to \pi_{\infty}=\mathcal N(0,1),
\end{equation}
in probability as $n\to+\infty$, on the space of probability measures on $\mathbb R$ equipped with the topology of weak convergence.
\end{proposition}
Stronger results are already known for certain cases of split trees: 
in particular, it is known that~\eqref{eq:cv_profile} holds almost surely 
in the case of the binary search tree \cite{CDJ01}, 
the random recursive tree \cite{MM17}, 
and the preferential attachment tree \cite{Katona05}.
The profile of the uniform random tree 
(considered by~\cite{Chapuy} in the context of Voronoi cells)
converges in distribution to the local time of a Brownian excursion (see~\cite{DG}).

\begin{remark}\label{rk:quenched_vs_annealed}
Note that Theorem~\ref{th:voronoi} holds in an averaged 
sense (or with respect to the joint law). One could imagine two quenched versions by (i) conditioning on the random split tree $(\tau_n)_{n\geq 1}$ or (ii) conditioning additionally also on the sequence of edge lengths $\bs L$. Since, in our proof, we use the central limit theorem for the sequence $\bs L$, our current methods do not provide with a possible version of Theorem~\ref{th:voronoi} quenched with respect to $\bs L$.
However, for the split distributions $\nu$ for which~\eqref{eq:cv_profile} holds almost surely, Theorem~\ref{th:voronoi} would hold almost surely given $(\tau_n)_{n\geq 1}$.
\end{remark}

The remainder of the paper is organised as follows. In Section \ref{sec:voronoi}, we establish a central limit theorem for the joint law of the height of uniform vertices and derive Proposition \ref{prop:profile}. Furthermore, we proof Theorem \ref{th:voronoi}. In Section \ref{sec:territories}, we extend these arguments to the case of different speeds thereby proving Theorem \ref{th:territories}.

\section{Proof of Theorem~\ref{th:voronoi}}\label{sec:voronoi}
In this section, we use the same notation, and place ourselves under the assumptions of Theorem~\ref{th:voronoi}. The idea of the proof is as follows: if $\varpi = \delta_1$ (i.e.\ all edge lengths are equal to~1 almost surely, i.e.\ $\bs L \equiv 1$) then among the nodes $U_1(n), \ldots, U_k(n)$, the one closest to the root belongs to the Voronoi cell containing the root, and this Voronoi cell typically is the largest of all Voronoi cells. 
As a consequence, it is important 
to understand the heights of $k$ uniform nodes in a random split tree. 
Recall that for a graph node $v\in\tau$, we write $|v|$ for that graph distance between the root $\varnothing$ and $v.$
\begin{lemma}[CLT for heights of uniform vertices]\label{lem:height_unif}
Let $k\in\mathbb N$  and $\bar Y$ be distributed as the size-biased version of the marginals of $\nu$ (see \eqref{eq:barY}), and denote $\mu = \mathbb E[\log (\nicefrac1{\bar Y})]$ as well as $\sigma^2 = \mathrm{Var}(\log \bar Y).$ 

Then, in distribution as $n\to+\infty$, we have 
\[\left(\frac{|U_1(n)| - (\log n)/\mu}{\sqrt{(\log n)/\mu^3}},\ldots,
\frac{|U_k(n)| - (\log n)/\mu}{\sqrt{(\log n)/\mu^3}}
\right)\Rightarrow (\Lambda_1, \ldots, \Lambda_k),
\]
where the $\Lambda_1, \ldots, \Lambda_k$ are independent centred Gaussian random variables of variance $\sigma^2.$
\end{lemma}

This lemma straightforwardly implies Proposition~\ref{prop:profile}. 
\begin{proof}[{Proof of Proposition~\ref{prop:profile}}] We use \cite[Lemma~3.1]{MM17}, which states that for a sequence of random measures $(\pi_n)_{n\geq 0}$ to converge in probability to a limiting measure $\pi_{\infty}$, it is enough to show, 
for two random variables $A_n$ and $B_n$ sampled independently according to the random measure $\pi_n$, that 
$(A_n, B_n) \to (A, B)$ in distribution, where $A$ and $B$ are $\nu$-distributed and independent. (Note that, on the left-hand side, $A_n$ and $B_n$ are independent conditionally on $\pi_n$, but not without this conditioning.)
As a direct consequence, we conclude the proof using  \cite[Lemma~3.1]{MM17} in combination with Lemma \ref{lem:height_unif} for the particular case $k=2$ in order to ensure the required convergence conditions of \cite[Lemma~3.1]{MM17} to be fulfilled.
\end{proof}

To prove Lemma~\ref{lem:height_unif}, we first prove convergence of the marginals and then derive asymptotic independence; the latter is a consequence of the following lemma:
\begin{lemma}\label{lem:LCA}
For all $h\in \mathbb N$, let us denote by $U^{\sss (h)}_1(n), \ldots, U^{\sss (h)}_k(n)$ the respective ancestors of $U_1(n), \ldots, U_k(n)$ that have height $h$ (if $h>|U_i(n)|$, we set $U^{\sss (h)}_i(n)=U_i(n)$). 
Let $H_n=\max_{1\le i < j \le k} |U_i(n)\wedge U_j(n)|$ be the height of the most recent common ancestor of $U_1(n),\dots,U_k(n)$ and $S_1(n), \ldots, S_k(n)$ be the sizes of the subtrees of $\tau_n$ rooted at $U_1^{\sss (H_n)}(n), \ldots, U_k^{\sss (H_n)}(n)$ respectively.
In distribution when $n\to+\infty$,
\[\Big(\frac{S_1(n)}n, \ldots, \frac{S_k(n)}n, H_n\Big)
\Rightarrow (\alpha_1, \ldots, \alpha_k, H),\]
where $H$ is an almost surely finite random variable, 
and $\alpha_1, \ldots, \alpha_k$ are almost surely positive random variables.
\end{lemma}

\begin{proof}
We first look at the last common ancestor of $U_1(n)$ and $U_2(n)$:
for all words $w\in\mathcal D_m$, 
we have
\[\mathbb P(U_1(n)\wedge U_2(n) = w\, |\, \tau_n)
= \sum_{1\leq i\neq j\leq m} \frac{s_{wi}(n)}n\cdot \frac{s_{wj}(n)}n,\]
where $s_v(n)$ is the size of the subtree of $\tau_n$ rooted at $v$ 
(in particular, this is equal to zero if $v\notin\tau_n$).
By the definition of the model and the strong law of large numbers we know that, conditionally on the sequence $\bs Y = (\bs Y(v))_{v\in\mathcal D_m}$, for all $v\in\mathcal D_m$, almost surely when $n\to+\infty$, 
\[
\frac{s_{v}(n)}n \to \prod_{u\preccurlyeq v} Z_u,
\]
where we recall that $Z_{w\ell} = Y_\ell(w)$, 
for all $w\in\mathcal D_m$ and $\ell\in\{1, \ldots, m\}$.
Therefore, conditionally on $\bs Y$ and almost surely when $n\to+\infty$, we have
\[\mathbb P(U_1(n)\wedge U_2(n) = w\, |\, \tau_n, \bs Y)
\to \sum_{1\leq i\neq j\leq m} \Big(\prod_{u\preccurlyeq wi}Z_u\Big) \Big(\prod_{u\preccurlyeq wj}Z_u\Big),\] 
which implies, using dominated convergence, \[\mathbb P(U_1(n)\wedge U_2(n) = w)
\to \mathbb E\bigg[\sum_{1\leq i\neq j\leq m} \Big(\prod_{u\preccurlyeq wi}Z_u\Big) \Big(\prod_{u\preccurlyeq wj}Z_u\Big)\bigg].
\]
To prove that this implies convergence in distribution of $U_1(n)\wedge U_2(n)$  to an almost surely finite random variable~$K_{1,2}\in\mathcal D_m$, we need to show that
\begin{equation}\label{eq:tightness}
\sum_{w\in\mathcal D_m}
\mathbb E\bigg[\sum_{1\leq i\neq j\leq m} \Big(\prod_{u\preccurlyeq wi}Z_u\Big) \Big(\prod_{u\preccurlyeq wj}Z_u\Big)\bigg] = 1.
\end{equation}
In order to prove \eqref{eq:tightness}, we first note that, by independence of the $Z_u$'s (except among siblings), 
for all $w\in\mathcal D_m$,
\[
\mathbb E\bigg[\sum_{1\leq i\neq j\leq m} \Big(\prod_{u\preccurlyeq wi}Z_u\Big) \Big(\prod_{u\preccurlyeq wj}Z_u\Big)\bigg]
= \mathbb E\bigg[\prod_{u\preccurlyeq w}Z^2_u\bigg]
\sum_{1\leq i\neq j\leq m} \mathbb E[Z_{wi}Z_{wj}]
= \beta\mathbb E\bigg[\prod_{u\preccurlyeq w}Z^2_u\bigg],
\]
where we have introduced the shorthand $\beta = \sum_{1\leq i\neq j\leq m} \mathbb E[Y_i Y_j]$,
with $Y$ a random vector of distribution~$\nu$
(note that, by Assumption (A1-i), $\beta\neq 0$).
This entails that
\begin{equation}\label{eq:sum_on_h}
\sum_{w\in\mathcal D_m}
\mathbb E\bigg[\sum_{1\leq i\neq j\leq m} \Big(\prod_{u\preccurlyeq wi}Z_u\Big) \Big(\prod_{u\preccurlyeq wj}Z_u\Big)\bigg] 
= \beta \sum_{w\in\mathcal D_m} \mathbb E\bigg[\prod_{u\preccurlyeq w}Z^2_u\bigg]
= \beta \sum_{h\geq 0} \sum_{|w|=h} \mathbb E\bigg[\prod_{u\preccurlyeq w}Z^2_u\bigg],
\end{equation}
where we recall that $|w|$ is the height of $w$.
For all $h\geq 0$, using again the independence of the $Z_u$'s, we infer that
\begin{equation}\label{eq:rec_h}
\sum_{|w|=h+1} \mathbb E\bigg[\prod_{u\preccurlyeq w}Z^2_u\bigg]
= \sum_{|w|=h} \sum_{i=1}^m \mathbb E\bigg[\prod_{u\preccurlyeq wi} Z^2_u\bigg]
= \sum_{|w|=h}\mathbb E\bigg[\prod_{u\preccurlyeq w}Z^2_u\bigg] \sum_{i=1}^m \mathbb E[Z^2_{wi}].
\end{equation}
Since, by definition, $\sum_{i=1}^m Z_{wi} = \sum_{i=1}^n Y_i(w) = 1$, we get that
\[\sum_{i=1}^m \mathbb E[Z^2_{wi}] 
= \Big(\sum_{i=1}^m \mathbb EZ_{wi}\Big)^2
- \sum_{1\leq i\neq j\leq m} \mathbb E[Z_{wi}Z_{wj}]
= 1-\beta,
\]
by definition of $\beta$.
Plugging the last equality into \eqref{eq:rec_h}, this amounts to
\[\sum_{|w|=h+1} \mathbb E\bigg[\prod_{u\preccurlyeq w}Z^2_u\bigg]
= (1-\beta)  \sum_{|w|=h}\mathbb E\bigg[\prod_{u\preccurlyeq w}Z^2_u\bigg]
= \dots = (1-\beta)^{h+1}
\]
by iteration, and further, using \eqref{eq:sum_on_h} and the fact that by Assumption (A1-i),
$\beta\neq 0$, we get
\[\mathbb E\bigg[\sum_{1\leq i\neq j\leq m} \Big(\prod_{u\preccurlyeq wi}Z_u\Big) \Big(\prod_{u\preccurlyeq wj}Z_u\Big)\bigg] 
 = \beta \sum_{h=0}^{\infty} (1-\beta)^h = 1.
\]
This concludes the proof of \eqref{eq:tightness}, and thus of the fact that $U_1(n)\wedge U_2(n)$ converges in distribution to an almost surely finite random variable~$K_{1,2}$,
Consequently, 
\[
H_n=\max_{1\le i< j\le k} |U_i(n)\wedge U_j(n)| \Rightarrow \max_{1\le i< j\le k} |K_{i,j}|=:H,
\]
where each of the $K_{i,j}$ (which are not independent) has the same distribution as $K_{1,2}.$ The random variable $H$ is almost surely finite since all the $K_{i,j}$'s are.

Finally, for all $n\geq 1$, $x_1, \ldots, x_k\in [0,\infty)$
and $h\in\{0, 1, \ldots\}$
we have 
\[\mathbb P\left(\frac{S_1(n)}n\geq x_1, \ldots, \frac{S_k(n)}n\geq x_k, H_n = h \,\Big|\,\tau_n, \bs Y\right)
=\sum_{w_1, \ldots, w_k\in\mathcal D_m^{\sss (h)}} 
\prod_{i=1}^k \frac{s_{w_i}(n)}n\,\bs 1\{s_{w_i}(n)\geq x_i n\},\]
where $\mathcal D_m^{\sss (h)}$ is the set of all distinct $w_1, \ldots, w_k\in\{1, \ldots, m\}^h$ such that the cardinality of the set $\{\parent w_1, \ldots, \parent w_k\}$ is at most $k-1$ (where we recall that $\parent w$ denotes the parent of a node $w$).
By the strong law of large numbers, we thus get
\[\mathbb P\left(\frac{S_1(n)}n\geq x_1, \ldots, \frac{S_k(n)}n\geq x_k, H_n = h \,\Big|\,\tau_n, \bs Y\right) \to \sum_{w_1, \ldots, w_k\in\mathcal D_m^{\sss (h)}} 
\prod_{i=1}^k \Big(\prod_{u\preccurlyeq w_i} Z_u\Big)\,\bs 1\Big\{\prod_{u\preccurlyeq w_i} Z_u\geq x_i \Big\},\]
and by dominated convergence, 
\[\mathbb P\left(\frac{S_1(n)}n\geq x_1, \ldots, \frac{S_k(n)}n\geq x_k, H_n = h\right) \to \mathbb E\bigg[\sum_{w_1, \ldots, w_k\in\mathcal D_m^{\sss (h)}} 
\prod_{i=1}^k \Big(\prod_{u\preccurlyeq w_i} Z_u\Big)\,\bs 1\Big\{\prod_{u\preccurlyeq w_i} Z_u\geq x_i \Big\}\bigg],
\]
which concludes the proof.

To see that $\alpha_j>0$ almost surely for all $j\in\{1, \ldots, k\}$, note that
\[\mathbb P(\alpha_j=0)
=\sum_{h\geq 0} \mathbb E\bigg[\sum_{w_1, \ldots, w_k\in\mathcal D_m^{\sss (h)}} 
\bs 1\Big\{\prod_{u\preccurlyeq w_j} Z_u=0\Big\}
\prod_{i=1}^k \Big(\prod_{u\preccurlyeq w_i} Z_u\Big)\bigg]
=0,\]
because, if $\prod_{u\preccurlyeq w_j} Z_u=0$ then $\prod_{i=1}^k \Big(\prod_{u\preccurlyeq w_i} Z_u\Big)=0$ and thus
\[\bs 1\Big\{\prod_{u\preccurlyeq w_j} Z_u=0\Big\}
\prod_{i=1}^k \Big(\prod_{u\preccurlyeq w_i} Z_u\Big)=0.\]
Since this is true for all $1\leq j\leq k$, we indeed have that the $\alpha_j$'s are almost surely positive.
\end{proof}

\begin{proof}[Proof of Lemma~\ref{lem:height_unif}]
We first prove the convergence of the marginals: let $k_n$ be an integer chosen uniformly at random in $\{1, \ldots, n\}$, then $\xi(k_n) = U_1(n)$ in distribution; recall that by definition, for all $n\geq 1$, $\xi(n)$ is the unique node that belongs to $\tau_n$ but not to $\tau_{n-1}$.
By~\cite[Th.\ 2]{devroye-split}, 
we know that, in distribution when $n\to+\infty$,
\[\frac{|\xi(n)|-(\log n)/\mu}{\sqrt{(\log n)/\mu^3}}\Rightarrow\mathcal N(0,\sigma^2).\]
Therefore, since $\log k_n = \log n + \mathcal \Op(1)$ when $n\to+\infty$, we get
\begin{equation}\label{eq:marginal}
\frac{|\xi(k_n)|-(\log n)/\mu}{\sqrt{(\log n)/\mu^3}}\Rightarrow\mathcal N(0,\sigma^2),
\end{equation}
which immediately entails the convergence of the marginals to the desired limit. 
To show that the limits are independent, we use Lemma~\ref{lem:LCA}, and the fact that, by definition of the model, given $H_n$, $S_1(n), \ldots, S_k(n)$, the trees rooted at $U_1^{\sss (H_n)}(n), \ldots, U_k^{\sss (H_n)}(n)$ are independent split trees of split distribution $\nu$ and of respective sizes $S_1(n), \ldots, S_k(n)$. Moreover, for all $1\leq i\leq k$, the node $U_i(n)$ is distributed uniformly at random among the nodes of the split tree rooted at $U_i^{\sss (H_n)}(n)$.
Therefore, given $H_n$, $S_1(n), \ldots, S_k(n)$, we have, in distribution and jointly for all $1\leq i\leq k$, 
\[|U_i(n)| = H_n + |\widehat U^{\sss (i)}(S_i(n))|,\]
where the $\widehat U^{\sss (i)}$'s are independent, and for all $i$, $\widehat U^{\sss (i)}(S_i(n))$ is a node taken uniformly at random in a split tree of size $S_i(n)$.
As a consequence, applying \eqref{eq:marginal} to each of the~$k$ independent split trees, we get that, in distribution and jointly for all $1\leq i\leq k$,
\begin{align*}
\frac{|U_i(n)|-(\log n)/\mu}{\sqrt{(\log n)/\mu^3}}
&= \frac{H_n +  |\widehat U^{\sss (i)}(S_i(n))|-(\log n)/\mu}{\sqrt{(\log n)/\mu^3}}\\
&= \frac{|\widehat U^{\sss (i)}(S_i(n))|-(\log S_i(n))/\mu}{\sqrt{(\log S_i(n))/\mu^3}}
\cdot \sqrt{\frac{\log S_i(n)}{\log n}} + \frac{\log S_i(n)-\log n}{\sqrt{\mu\log n}}
+\op(1)\\
&\Rightarrow \Lambda_i,
\end{align*}
where $(\Lambda_1,\ldots, \Lambda_k)$ are $k$ independent centred Gaussians of variance $\sigma^2$;
we have used the fact that by Lemma~\ref{lem:LCA}, 
$\log S_i(n) = \log n + \Op(1)$ when $n\to+\infty$.
\end{proof}

Applying the law of large numbers to the i.i.d.\ edge lengths, and using the fact that, by Lemma~\ref{lem:LCA}, the height $H_n$ of the last common ancestor of $U_1(n),\ldots, U_k(n)$ converges in distribution to an almost surely finite random variable, Lemma~\ref{lem:height_unif} entails a similar result for the distances of $U_1(n),\ldots, U_k(n)$ to the root (for the distance $d_{\bs L}$).
In the following, $L$ denotes a random variable distributed according to $\varpi.$
\begin{lemma}[CLT for distances to the root]\label{lem:real_heights_finite_var}
Under the assumptions of Theorem~\ref{th:voronoi},
if $\mathrm{Var}(L)<+\infty$, then, in distribution when $n\to+\infty$, we have
\[\left(\frac{|U_1(n)|_{\bs L} - (\log n)\mathbb EL/\mu}{\sqrt{(\log n)/\mu^3}},\ldots,
\frac{|U_k(n)|_{\bs L} - (\log n)\mathbb EL/\mu}{\sqrt{(\log n)/\mu^3}}
\right)\Rightarrow (\Xi_1, \ldots, \Xi_k),
\]
where the $\Xi_1, \ldots, \Xi_k$ 
are independent centred Gaussian random variables
of variance $\mathrm{Var}(L)+\sigma^2(\mathbb EL)^2$.
\end{lemma}

\begin{proof}
In this proof, we set $U_i[H_n] = U_i^{\sss (H_n)}(n)$, i.e.\ the ancestor of $U_i(n)$ at height $H_n$,
where $H_n$ is defined in Lemma~\ref{lem:LCA}.
For all $1\leq i\leq k$, we have
\begin{equation}\label{eq:Ai}
|U_i(n)|_{\bs L} = \sum_{\varnothing\neq u\preccurlyeq U_i(n)} L_u
= \sum_{\varnothing \neq u \preccurlyeq U_i[H_n]} \big(L_u - L_u^{\sss (i)}\big)
+ A_i(n),
\end{equation}
where the $\bs L^{\sss (i)} = (L_u^{\sss (i)})_{u\in\mathcal D_m}$ are i.i.d.\ copies of $\bs L$, and where 
\[A_i(n):=\sum_{\varnothing\neq u\preccurlyeq U_i[H_n]} L^{\sss (i)}_u
+\sum_{U_i[H_n]\prec u \preccurlyeq U_i(n)} L_u.\]
Since, by Lemma~\ref{lem:LCA}, $H_n$ converges in distribution to an almost surely finite random variable $H$, we have
\begin{equation}\label{eq:LLN_change_dist}
\frac1{\sqrt{\log n}}
\sum_{\varnothing \neq u \preccurlyeq U_i[H_n]}\big(L_u-L_u^{\sss (i)}\big) 
\to 0,
\end{equation}
in distribution when $n\to+\infty$.
Note that, given $H_n$, the random variables $A_1(n), \ldots, A_k(n)$ are independent, 
because the $L_u$'s are independent, 
and the sums in $A_1(n), \ldots, A_k(n)$ that involve the sequence $\bs L$ 
(as opposed to its i.i.d.\ copies $\bs L^{\sss (1)}, \ldots, \bs L^{\sss (k)}$) 
range over distinct nodes~$u$. Therefore, in distribution, we have, jointly for all $1\leq i\leq k$,
\[A_i(n)= \sum_{\varnothing\neq u\preccurlyeq U_i(n)} L^{\sss (i)}_u
= \sum_{j=1}^{|U_i(n)|} \widetilde L^{\sss (i)}_j,\]
where $\bs L^{\sss (i)} = (\widetilde L_j^{\sss (i)})_{j\geq 1}$ is a sequence of i.i.d.\ copies of the $L^{\sss (i)}_u$'s, and the $k$ sequences $(\bs L^{\sss (i)})_{1\leq i\leq k}$ are independent of each other.
By the central limit theorem, we have, jointly for all $1\leq i\leq k$,
\[\frac{\sum_{j=1}^m \widetilde L^{\sss (i)}_j - m\mathbb EL}{\sqrt m} \Rightarrow  \Theta_i,\]
in distribution when $n\to+\infty$, 
where $\Theta_1, \ldots, \Theta_k$ are independent standard Gaussians.
Since $|U_i(n)|\to+\infty$ in probability when $n\to+\infty$, and since $(|U_i(n)|)_{1\leq i\leq k}$  is independent from $(L_j^{\sss (i)} : j\geq 1)_{1\leq i\leq k}$, this implies that, jointly for all $1\leq i\leq k$,
\begin{equation}\label{eq:CLT_A_i}
\frac{A_i(n)-|U_i(n)|\mathbb EL}{\sqrt{\mathrm{Var}(L)|U_i(n)|}} 
\Rightarrow \Theta_i.
\end{equation}
Indeed, for all $u_1, \ldots, u_k\in\mathbb R$, and $\varepsilon>0$,
there exists $m_0>0$ such that, for all $m_1, \ldots, m_k\geq m_0$,
\[\bigg|\mathbb P\left(\frac{\sum_{j=1}^{m_i} \widetilde L^{\sss (i)}_j - m_i\mathbb EL}{\sqrt m_i}\leq u_i, \forall 1\leq i\leq k\right) - \mathbb P(\Theta_i\leq u_i, \forall 1\leq i\leq k)\bigg|\leq \nicefrac\varepsilon2.\]
Because $U_i(n)\to+\infty$ in probability, there exists $n_0$ such that, for all $n\geq n_0$,
$\mathbb P(\inf_{1\leq i\leq k} |U_i(n)|\geq m_0) \geq 1-\nicefrac\varepsilon2$.
Therefore, for all $u_1, \ldots, u_k\in\mathbb R$, $\varepsilon>0$, and $n\geq n_0$,
\begin{align*}
&\bigg|\mathbb P\left(\frac{A_i(n) - |U_i(n)|\mathbb EL}{\sqrt |U_i(n)|}\leq u_i, \forall 1\leq i\leq k\right) - \mathbb P(\Theta_i\leq u_i, \forall 1\leq i\leq k)\bigg|\\
&\quad\leq \bigg|\mathbb P\left(\frac{A_i(n) - |U_i(n)|\mathbb EL}{\sqrt |U_i(n)|}\leq u_i, \forall 1\leq i\leq k \text{ and }\inf_{1\leq i\leq k}|U_i(n)|\geq m_0\right) - \mathbb P(\Theta_i\leq u_i, \forall 1\leq i\leq k)\bigg| + \frac\varepsilon2\\
&\quad \leq \sum_{\ell = m_0}^\infty \bigg|\mathbb P\left(\frac{A_i(n) - \ell\mathbb EL}{\sqrt \ell}\leq u_i, \forall 1\leq i\leq k\right) - \mathbb P(\Theta_i\leq u_i, \forall 1\leq i\leq k)\bigg| \mathbb P\Big(\inf_{1\leq i\leq k}|U_i(n)|=\ell\Big) +  \frac\varepsilon2\leq \varepsilon,
\end{align*}
where, in the second inequality, we have conditioned on the different possible values of~$\inf_{1\leq i\leq k}|U_i(n)|$, and used the triangular inequality.
This concludes the proof of~\eqref{eq:CLT_A_i}.

We thus get that, jointly for all $1\leq i\leq k$,
\begin{align*}
\frac{|U_i(n)|_{\bs L}-(\log n)\mathbb EL/\mu}{\sqrt{(\log n)/\mu^3}}
&= \frac{|U_i(n)|_{\bs L}-|U_i(n)|\mathbb EL}{\sqrt{\mathrm{Var}(L)|U_i(n)|}}
\cdot \sqrt{\frac{\mathrm{Var}(L)|U_i(n)|}{(\log n)/\mu^3}}
+ \frac{|U_i(n)|\mathbb EL-(\log n)\mathbb EL/\mu}{\sqrt{(\log n)/\mu^3}}\\
&\Rightarrow \Theta_i\sqrt{\mathrm{Var}(L)}+\Lambda_i\mathbb EL,
\end{align*}
where we have used Lemma~\ref{lem:height_unif} 
(where $\Lambda_1, \ldots, \Lambda_k$ are defined).
Note that, by definition, $\Theta_i$ is independent from $\Lambda_i$ for all $1\leq i\leq k$; 
indeed, $\Theta_i$ is $\bs L^{\sss (i)}$-measurable, 
while $\Lambda_i$ is $(\tau_n)_{n\geq 1}$-measurable.
Therefore, $\Xi_i:=\Theta_i\sqrt{\mathrm{Var}(L)}+\Lambda_i\mathbb EL$ is a centred Gaussian of variance $\mathrm{Var}(L) + \sigma^2\mathbb E(L)^2$, and $\Xi_1, \ldots, \Xi_k$ are independent, as claimed.
\end{proof}

We now look at the respective version of Lemma \ref{lem:real_heights_finite_var} when the edge lengths have heavier tails:
\begin{lemma}\label{lem:real_height_heavy_tails}
Under the assumptions of Theorem~\ref{th:voronoi},
if there exists $\alpha\in (1,2)$ and a function $\ell\colon[0,+\infty)\to[0,+\infty)$ slowly varying at infinity such that $\mathbb P(L\geq x)=x^{-\alpha}\ell(x)$ for all $x\geq 0$, then, 
in distribution when $n\to+\infty$,
\[\left(\frac{|U_1(n)|_{\bs L} - (\log n)\mathbb EL/\mu}{((\log n)/\mu)^{\nicefrac1\alpha}},\ldots,
\frac{|U_k(n)|_{\bs L} - (\log n)\mathbb EL/\mu}{((\log n)/\mu)^{\nicefrac1\alpha}}
\right)\Rightarrow (\Upsilon_1(\alpha), \ldots, \Upsilon_k(\alpha)),
\]
where the $\Upsilon_1(\alpha), \ldots, \Upsilon_k(\alpha)$ are i.i.d.\ copies of a centred $\alpha$-stable random variable.
\end{lemma}

\begin{proof}
We proceed as in the proof of Lemma~\ref{lem:real_heights_finite_var} and employ the same notation:
in particular, using that our assumptions on $L$ entail its expectation being finite, \eqref{eq:Ai} and~\eqref{eq:LLN_change_dist} give that, in distribution when $n\to+\infty$, 
\begin{equation}\label{eq:UtoA}
|U_i(n)|_{\bs L} =  A_i(n)+o_{\mathbb P}(\sqrt{\log n}) = A_i(n)+\op((\log n)^{\nicefrac1\alpha}).
\end{equation} 
Using the functional limit theorem for sums of i.i.d.\ heavy-tailed random variables 
(see, e.g., Theorem 2 in \cite[\S\ 35]{GK}), we have that
\[\frac{A_i(n)-|U_i(n)|\mathbb EL}{|U_i(n)|^{\nicefrac1\alpha}}\Rightarrow \Upsilon_i(\alpha),\]
where $\Upsilon_1(\alpha), \ldots, \Upsilon_k(\alpha)$ are i.i.d.\ copies of an $\alpha$-stable random variable.
We thus get that
\[\frac{|U_i(n)|_{\bs L} - |U_i(n)|\mathbb EL}{|U_i(n)|^{\nicefrac1\alpha}}
= \frac{A_i(n)-|U_i(n)|\mathbb EL}{|U_i(n)|^{\nicefrac1\alpha}} 
+ \frac{|U_i(n)|_{\bs L}-A_i(n)}{((\log n)/\mu)^{\nicefrac1\alpha}}
\cdot\frac{((\log n)/\mu)^{\nicefrac1\alpha}}{|U_i(n)|^{\nicefrac1\alpha}}
\Rightarrow \Upsilon_i(\alpha),
\]
in distribution when $n\to+\infty$, where we have used \eqref{eq:UtoA} and the fact that, by Lemma~\ref{lem:height_unif}, $\frac{|U_i(n)|}{(\log n)/\mu} \to 1$ in probability when $n\to+\infty$.
We thus get that, jointly for all $1\leq i\leq k$,
\begin{align}
\frac{|U_i(n)|_{\bs L}-(\log n)\mathbb EL/\mu}{((\log n)/\mu)^{\nicefrac1\alpha}}
&= \frac{|U_i(n)|_{\bs L}-|U_i(n)|\mathbb EL}{|U_i(n)|^{\nicefrac1\alpha}}
\cdot \bigg(\frac{|U_i(n)|}{(\log n)/\mu}\bigg)^{\!\!\nicefrac1\alpha}
+ \frac{|U_i(n)|\mathbb EL-(\log n)\mathbb EL/\mu}{((\log n)/\mu)^{\nicefrac1\alpha}}\label{eq:remark}\\
&\Rightarrow\Upsilon_i(\alpha),\notag
\end{align}
in distribution when $n\to+\infty$; here, we have used Lemma~\ref{lem:height_unif} and the fact that $\alpha<2$, which implies $\nicefrac1\alpha>\nicefrac12$ (and $\alpha > 1$ again to get the finiteness of $\mathbb E L$).
\end{proof}

\begin{remark}
Note that in the $\alpha$-stable case, the second summand on the right-hand side of~\eqref{eq:remark} is negligible compared to the first summand, i.e.\ the fluctuations coming from the height of $U_i(n)$ are asymptotically negligible in front of the fluctuations coming from the edge-lengths.
\end{remark}

Next we control the sizes of subtrees rooted at certain nodes within the tree. For this purpose, imagine that $U_1(n)$ is the closest to the root (in graph distance) among $U_1(n), \ldots, U_k(n)$. 
Then the Voronoi cell of $U_2(n)$ is the subtree rooted at the ancestor of $U_2(n)$ 
that has height 
\[|U_1(n)\wedge U_2(n)|_{\bs L} + \big||U_1(n)|-|U_2(n)|\big|_{\bs L}/2.\] 
From Lemma~\ref{lem:LCA}, we already know that $|U_1(n)\wedge U_2(n)|$ converges in distribution to an almost surely finite random variable.
The following lemma gives a limiting result for the size of the subtree rooted at an ancestor 
of $U_2(n)$ at height $h(n)$ for some function $h(n) = o(\log n)$ when $n\to+\infty$.


For $f\colon \mathbb N\to \mathbb R$, $x\in[0, \infty)$, $n\geq 1$, and $1\leq i\leq k$, 
we write $D^{\sss (i)}_n(xf)$ (or $D^{\sss (i)}_n(xf(n))$) for the size of the subtree of $\tau_n$ rooted at the ancestor of $U_i(n)$ closest to the root whose $\bs L$-distance to the root is at least
\begin{equation} \label{eq:Fdef}
xF_i(n):=\min(xf(n), |U_i(n)|_{\bs L}).
\end{equation}
By Lemmas~\ref{lem:real_heights_finite_var} and~\ref{lem:real_height_heavy_tails}, 
$|U_i(n)|_{\bs L}$ grows logarithmically in $n$;
therefore, if $f(n)=o(\log n)$, then for large~$n$ we typically have $F_i(n)=f(n)$.

\begin{proposition}[Convergence of (extended) fringe trees]\label{prop:fringe_trees}
Let $f\colon \mathbb N\to \mathbb N$ be a function such that $\lim_{n\to+\infty}f(n)=+\infty$, and $x> 0$. 
We assume that either $f(n) = o(\log n)$ 
when $n\to+\infty$, or $f(n) = \log n$ for all $n\geq 1$ and set \[C(f)=\begin{cases}+\infty &\text{if }f(n)=o(\log n),\\\mathbb EL/\mu&\text{if }f(n)=\log n.\end{cases}\]
Then, under the assumptions of Theorem~\ref{th:voronoi}, 
for all $1\leq i\leq k$,
for all $0<a\leq b<C(f)$,
\begin{equation} \label{eq:fringe_trees}
\sup_{x\in[a,b]} \Big|\frac{\log (D_n^{\sss (i)}(xf)/n)}{xf(n)}- \frac{\mathbb E[\log \bar Y]}{\mathbb EL}\Big|\to 0 
\qquad\text{in probability when $n\to+\infty$.}
\end{equation}
\end{proposition}

This lemma is at the heart of the proof of Theorem~\ref{th:voronoi}: it establishes a law of large numbers for the logarithm of the size of fringe trees. 
A fringe tree, as defined in~\cite{JansonHolm} (see also the references therein for a literature review on the subject), is the subtree rooted at a node taken uniformly at random among the $n$-nodes of a random tree (in our case the $n$-node split tree of split distribution $\nu$). 
An extended fringe tree (still following~\cite{JansonHolm}) 
is the subtree rooted at one of the ancestors of this randomly chosen node, under the assumption that this ancestor is at a \emph{fixed} graph distance of the randomly chosen node.
In Proposition~\ref{prop:fringe_trees}, however, we also consider subtrees rooted 
at an ancestor of a node $U(n) = U_i(n)$ chosen uniformly at random in our tree, 
but this ancestor can be at a distance 
that grows with $n$. 
Therefore, we get results that are weaker than the result stated in~\cite{JansonHolm} in the case of the binary search tree and the random recursive tree (which, we recall, are both split trees).
 
\begin{proof}[Proof of Proposition \ref{prop:fringe_trees}]
We defined the split tree $(\tau_n)_{n\geq 1}$ as a sequence of random trees, with $\xi(n+1)$ denoting the unique node in $\tau_{n+1}$ but not in $\tau_n$. 
Now let $\widetilde U$ be a uniform random variable on $[0,1]$, 
set $k_n = \lceil \widetilde Un\rceil\in \{1, \ldots, n\}$ for all $n\geq 1,$ and note that 
\begin{equation} \label{eq:n-kInfty}
n-k_n\to+\infty \qquad \text{almost surely as } n\to+\infty.
\end{equation}
{We fix $i\in\{1,\dots,k\}$ throughout the proof.} Letting $U_i(n)$ be the node of index $k_n$, we observe that $U_i(n)$ is indeed uniformly distributed in $\tau_n$, as required.

We fix $0<a< b<C(f)$. For $x\in[a,b]$, we define $h=h(n)$ so that $xh(n)$ 
is the height of the ancestor of $U_i(n)$ closest to the root whose $\bs L$-distance 
to the root is at least $xF_i(n)=\min(xf(n), |U_i(n)|_{\bs L})$ (recall \eqref{eq:Fdef}). 
As a consequence of Lemmas \ref{lem:real_height_heavy_tails} and \ref{lem:real_heights_finite_var}, respectively, $F_i(n)\to+\infty$ in probability when $n\to+\infty$, therefore
\begin{equation}\label{eq:h_vs_f}
\frac{F_i(n)}{h(n)} \to \mathbb EL\quad\text{ in probability when }n\to+\infty
\end{equation}
by the law of large numbers. 

Recall that 
$U_i^{\sss (xh(n))}(n)$ denotes the ancestor of $U_i(n)$ closest to the root whose height is at least $xh(n)$; 
we denote by $k_n(x)$ the integer such that $U_i^{\sss (xh(n))}(n) = \xi(k_n(x))$. 
By definition, we have $k_n(x)\leq k_n$. 
We next derive a law of large numbers for the width of the split interval associated to the node of index $k_n(x)$. Recall that, by definition of $(\tau_n)_{n\geq 1}$, to each node $w \in \tau_n$ (among which $\xi(k_n)$) is associated a sub-interval of $[0,1]$, whose length is given by $\prod_{\varnothing\neq u\preccurlyeq w} Z_u.$ We let 
\begin{equation} \label{eq:QnDef}
Q_n(x) := \prod_{\varnothing\neq u\preccurlyeq \xi(k_n(x))} Z_u
\end{equation} 
be the length of the interval associated to node $\xi(k_n(x))$.
We claim the following law of large numbers for $Q_n(x)$:
\begin{equation}\label{eq:LLN2}
\frac{\log Q_n(x)}{xh(n)}
\to \mathbb E[\log \bar Y] = -\mu \qquad \text{ in probability when $n\to+\infty$.}
\end{equation}
Indeed, first note that the random variables $Z_u$ 
are sized-biased since we condition on the event  $u\preccurlyeq \xi(k_n)$; 
more precisely, we condition the intervals associated to the nodes $u$ occurring in the product of \eqref{eq:QnDef} to contain $X_{k_n}$. In other words, conditionally on $u\preccurlyeq \xi(k_n)$, we have $Z_u = \bar Y$ in distribution, where  $\bar Y$ as in \eqref{eq:barY} and $\bs Y = (Y_1, \ldots, Y_m)\sim \nu$. 
Therefore, by the law of large numbers, since $h(n)\to+\infty$ in probability 
(cf.\ \eqref{eq:h_vs_f}), we get
\[\frac{\log Q_n(x)}{xh(n)}
= \frac1{xh(n)}\log \bigg(\prod_{\varnothing\neq u\preccurlyeq \xi(k_n(x))} Z_u\bigg)
= \frac1{xh(n)}\sum_{\varnothing\neq u\preccurlyeq \xi(k_n(x))} \log Z_u
\to \mathbb E[\log \bar Y],\]
in probability when $n\to+\infty$, which concludes the proof of~\eqref{eq:LLN2}.

\medskip
The main step now is to show 
\begin{equation}\label{eq:h_instead_f}
\sup_{x\in[a,b]} \Big|\frac{\log (D_n^{\sss (i)}(xf)/n)}{xh(n)}- \mathbb E[\log \bar Y]\Big|\to 0 
\qquad\text{in probability when $n\to+\infty$,}
\end{equation}
At the end of the proof, we show that this implies~\eqref{eq:fringe_trees}.

To prove~\eqref{eq:h_instead_f}, we rewrite its left-hand side as a sum of several terms to which we will apply various concentration inequalities. First note that, for all $x\in [a,b]$, for all $n\geq 1$,
\begin{align}
&\frac1{xh(n)}\Big|\log \Big(\frac{D_n^{\sss (i)}(xf)}n\Big)-xh(n)\mathbb E[\log \bar Y]\Big|\notag\\
&\qquad\leq \frac1{ah(n)}\Big|\log \frac{D_n^{\sss (i)}(xf)}{(n-k_n(x))Q_n(x)}\Big|
+\frac1{ah(n)}\Big|\log\Big(1-\frac{k_n(x)}n\Big)\Big|
+\Big|\frac{\log Q_n(x)}{xh(n)}-\mathbb E[\log \bar Y]\Big|.\label{eq:three_terms}
\end{align}
We show that the right hand side is $\op(1)$ as $n\to+\infty$ \emph{uniformly for $x\in[a,b]$} by treating each of the three summands separately. 

We start with the second term, which is the easiest. Recall that $k_n(x)\leq k_n$ and that $k_n = \lceil \widetilde Un\rceil$, so
\begin{equation}\label{eq:1st}
\sup_{x\in[a,b]}\Big|\log\Big(1-\frac{k_n(x)}n\Big)\Big| \leq \Big|\log\Big(1-\frac{k_n}n\Big)\Big|
\to -\log (1-\widetilde U),
\end{equation}
almost surely as $n\to+\infty$.

For the third term on the right-hand side of~\eqref{eq:three_terms}, we proceed as follows. 
In distribution, $(Q_n(x))_{x\in[a,b]} = (\sum_{\ell\leq xh(n)}\log \widetilde Z_\ell)_{x\in[a,b]}$, 
where $(\widetilde Z_\ell)_{\ell\geq 1}$ is a sequence of i.i.d.\ copies of $\bar Y$.
We apply the law of the iterated logarithm to the sequence 
of i.i.d.\ random variables $(\log \widetilde Z_\ell)_{\ell\geq 1}$, 
this gives
\[\limsup_{m\to+\infty}  \frac{|\sum_{\ell=1}^m \log \widetilde Z_\ell - m\mathbb E[\log\bar Y]|}{\sqrt{m\log\log m}} = \mathrm{Var}(\log\bar Y)^{\nicefrac12}= \sigma\quad\text{ almost surely}.\]
This implies in particular that there exists an almost surely finite random number $m_0$ such that almost surely,
\[\sup_{m\geq m_0}  \frac{|\sum_{\ell=1}^m \log \widetilde Z_\ell - m\mathbb E[\log\bar Y]|}
{\sqrt{m\log\log m}}
\leq 2\sigma.\]
We fix $\varepsilon>0$ and choose $m_1\geq m_0$ such that 
$2\sigma\sqrt{\log\log(m_1)/m_1}\leq \varepsilon$.
We have, almost surely
\begin{align*}
\sup_{m\geq m_1} \Big|\frac{\sum_{\ell=1}^m \log \widetilde Z_\ell}m -\mathbb E[\log\bar Y]\Big|
&\leq \sqrt{\frac{\log\log m_1}{m_1}}
\sup_{m\geq m_1} \frac{|\sum_{\ell=1}^m \log \widetilde Z_\ell - m\mathbb E[\log\bar Y]|}{\sqrt{m\log\log m}}\\
&\leq\sqrt{\frac{\log\log m_1}{m_1}} \sup_{m\geq m_0}  \frac{|\sum_{\ell=1}^m \log \widetilde Z_\ell - m\mathbb E[\log\bar Y]|}
{\sqrt{m\log\log m}}\leq \varepsilon.
\end{align*}
Consequently, we have for all $n\geq 1$ that 
\begin{align*}
&\mathbb P\left(\sup_{x\in[a,b]}  \Big|\frac{\log Q_n(x)}{xh(n)} - \mathbb E[\log\bar Y]\Big|\geq \varepsilon\right)
= \mathbb P\left(\sup_{x\in[a,b]}  \Big|\frac{\sum_{\ell=1}^{xh(n)} \log \widetilde Z_\ell}{xh(n)}- \mathbb E[\log\bar Y]\Big|\geq\varepsilon\right)\\
&\qquad\leq \mathbb P\left(\sup_{x\in[a,b]}  \Big|\frac{\sum_{\ell=1}^{xh(n)} \log \widetilde Z_\ell}{xh(n)}- \mathbb E[\log\bar Y]\Big|\geq\varepsilon\text{ and }ah(n)\geq m_1\right) + \mathbb P(ah(n)<m_1)\\
&\qquad\leq \mathbb P\left(\sup_{m\geq m_1}  \Big|\frac{\sum_{\ell=1}^m \log \widetilde Z_\ell}{m}- \mathbb E[\log\bar Y]\Big|\geq\varepsilon\right) + \mathbb P(ah(n)<m_1)
= \mathbb P(ah(n)<m_1)\to 0,
\end{align*}
as $n\to+\infty$, because $h(n)\to+\infty$ in probability with~$n$.
In other words,
\begin{equation}\label{eq:2nd}
\sup_{x\in [a,b]} \Big|\frac{\log Q_n(x)}{xh(n)}-\mathbb E[\log \bar Y]\Big|
=\op(1)\qquad\text{ as }n\to+\infty.
\end{equation}

Finally, we deal with the first term in the right-hand side of~\eqref{eq:three_terms} and aim to prove 
\begin{equation}\label{eq:3rd}
	\sup_{x\in[a,b]}\Big|\log \frac{D_n^{\sss (i)}(xf)}{(n-k_n(x))Q_n(x)}\Big| = \op(1)\quad\text{ as }n\to+\infty,
\end{equation}
Mind that inserting \eqref{eq:1st}, \eqref{eq:2nd} and \eqref{eq:3rd} into \eqref{eq:three_terms} implies \eqref{eq:h_instead_f}.
In order to prove \eqref{eq:3rd}, we first recall that, for all $x\in[a,b]$ conditionally on $k_n(x)$ and $\bs Z = (Z_u)_{u\in\mathcal D_m}$, we have, in distribution, 
\[D_n^{\sss (i)}(xf) = \sum_{\ell=k_n(x)+1}^n \bs1_{X_\ell < Q_n(x)}.\]
Thus, for all $x\in [a,b]$, $\varepsilon\in(0,1)$, $\lambda\geq 0$, the exponential Chebychev inequality yields
\begin{align*}
&\mathbb P\bigg(\sum_{\ell=k_n(x)}^n \bs1_{X_\ell < Q_n(x)}\geq (1+\varepsilon)Q_n(x)\big(n-k_n(x)\big)\,\Big|\, k_n(x), \bs Z\bigg)\\
&\leq \mathrm e^{-\lambda(1+\varepsilon)Q_n(x)(n-k_n(x))}
\prod_{\ell=k_n(x)}^n\big(1+Q_n(x)(\mathrm e^\lambda-1)\big)
= \exp\big(-Q_n(x)(n-k_n(x))((1+\varepsilon) \lambda -\mathrm e^{\lambda}+1)\big).
\end{align*}
Taking $\lambda = \log(1+\varepsilon)$, this yields the upper bound
\[\mathbb P\bigg(\sum_{\ell=k_n(x)}^n \bs1_{X_\ell < Q_n(x)}\geq (1+\varepsilon)Q_n(x)\big(n-k_n(x)\big)\,\Big|\, k_n(x), \bs Z\bigg)\leq\exp\big(-Q_n(x)(n-k_n(x))\tilde\varepsilon'\big),
\]
where we have set $\varepsilon':=(1+\varepsilon)\log(1+\varepsilon) - \varepsilon>0$. Using the fact that, for all $x\in[a,b]$, $Q_n(x)\geq Q_n(b)$, and $k_n(x)\leq k_n$, taking expectations on both sides, and then a supremum over $x\in[a,b]$, we infer that
\[\sup_{x\in[a,b]} 
\mathbb P\bigg(\sum_{\ell=k_n(x)}^n \bs1_{X_\ell < Q_n(x)}\geq (1+\varepsilon)Q_n(x)\big(n-k_n(x)\big)\bigg)\leq\mathbb E\big[\exp\big(-Q_n(b)(n-k_n)\tilde\varepsilon'\big)\big]
\]
In a similar vein, one deduces
\[\sup_{x\in[a,b]}\mathbb P\bigg(\sum_{\ell=k_n(x)}^n \bs1_{X_\ell < Q_n(x)}\leq (1-\varepsilon)Q_n(x)\big(n-k_n(x)\big)\bigg)\leq\mathbb E\big[\exp\big(-Q_n(b)(n-k_n)\varepsilon''\big)\big],\]
where $\varepsilon''=(1-\varepsilon)\log(1-\varepsilon)+\varepsilon.$
In total, we thus get that, for all $\varepsilon>0$,
\begin{equation}\label{eq:LDP}
\sup_{x\in[a,b]}\mathbb P\bigg(\big|D_n^{\sss (i)}(xf)-(n-k_n(x))Q_n(x)\big|>\varepsilon Q_n(x)\big(n-k_n(x)\big)\bigg)
\leq 2\mathbb E\big[\exp\big(-c\,Q_n(b)(n-k_n)\big)\big],
\end{equation}
where $c:=\min(\varepsilon', \varepsilon'')>0$.

We now recall that (see~\eqref{eq:h_vs_f}) $h(n) \sim F_i(n)/\mathbb EL$ in probability as $n\to+\infty$.
In the case when $f(n) = o(\log n)$, since $|U_i(n)|_{\bs L}/\log n$ converges to $\mu/\mathbb EL$ in probability when $n\to+\infty$, and 
since $xF_i(n) = \min(xf(n), |U_i(n)|_{\bs L})$, we get 
\begin{equation}\label{eq:cvprob_Fi/f}
\frac{F_i(n)}{f(n)}\to 1 \quad \text{ in probability as }n\to+\infty.
\end{equation}
In the case when $f(n) = \log n$, since $|U_i(n)|_{\bs L}/\log n$ converges to $\mu/\mathbb EL$ in probability when $n\to+\infty$, and since $x\leq b<\mu/\mathbb EL$, we get $F_i(n)\sim \log n$ in probability when $n\to+\infty$. 
In both cases, 
\begin{equation}\label{eq:cvprob_h/f}
\frac{h(n)}{f(n)}\to \frak c := \frac1{\mathbb EL}\quad \text{ in probability as }n\to+\infty.
\end{equation}
For all $\delta>0$ and for all $\eta>0$ small enough such that $\sup_{x\in [1-\eta, 1+\eta]}|\log x|\leq \delta$,
we have 
\begin{align}
&\mathbb P\bigg(\sup_{x\in[a,b]}\Big|\log \frac{D_n^{\sss (i)}(xf)}{(n-k_n(x))Q_n(x)}\Big|\geq\delta
\bigg)\notag\\
&\qquad\leq
\mathbb P\bigg(\sup_{x\in[a,b]}\Big|\frac{D_n^{\sss (i)}(xf)}{(n-k_n(x))Q_n(x)}-1\Big|\geq\eta
\text{ and } \frak c-\eta \leq \frac{h(n)}{f(n)}\leq \frak c+\eta\bigg)
+ \mathbb P\Big(\Big|\frak c-\frac{h(n)}{f(n)}\Big|>\eta\Big),\label{eq:prob}
\end{align}
where we have set $\frak c = 1/\mathbb EL$.
Because of~\eqref{eq:cvprob_h/f}, the second summand converges to~0 as $n\to+\infty$.
Note that, as $x$ increases between $a$ and $b$, ${D_n^{\sss (i)}(xf)}/{(n-k_n(x))Q_n(x)}$ only changes value when $xh(n)$ crosses an integer value.
Thus, the event inside the first probability on the right-hand side implies that there exists $\ell\in\mathbb N\cap[a(\frak c-\eta)f(n), b(\frak c+\eta)f(n)]$ such that $|{D_n^{\sss (i)}(xf)}/{(n-k_n(x))Q_n(x)}-1|\geq \eta$ for $x = \ell/h(n)$.
We thus get via a union bound 
\begin{align}
&\mathbb P\bigg(\sup_{x\in[a,b]}\Big|\frac{D_n^{\sss (i)}(xf)}{(n-k_n(x))Q_n(x)}-1\Big|\geq\eta
\text{ and } \frak c-\eta \leq \frac{h(n)}{f(n)}\leq \frak c+\eta\bigg)\notag\\
&\qquad\leq \sum_{\ell=\lceil a(1-\eta)f(n)\rceil}^{\lfloor b(1+\eta)f(n)\rfloor}
\mathbb P\bigg(\Big|\frac{D_n^{\sss (i)}(xf)}{(n-k_n(x))Q_n(x)}-1\Big|\geq\eta
\text{ for } x=\ell/h(n)\text{ and } \frak c-\eta \leq \frac{h(n)}{f(n)}\leq \frak c+\eta\bigg)\notag\\
&\qquad\leq ([b(\frak c+\eta)-a(\frak c-\eta)]f(n)+1)
\sup_{x\in [\hat a, \hat b]} \mathbb P\bigg(\Big|\frac{D_n^{\sss (i)}(xf)}{(n-k_n(x))Q_n(x)}-1\Big|\geq \eta\bigg)\notag\\
&\qquad\leq Kf(n)\mathbb E\big[\exp\big(-c\,Q_n(\hat b)(n-k_n)\big)\big],\label{eq:K}
\end{align}
where $\hat a = a(\frak c-\eta)/(\frak c-\eta)$ and $\hat b = b(\frak c+\eta)/(\frak c-\eta)$, and we used~\eqref{eq:LDP} in the last inequality.
We have also set $K$ large enough so that, for all $n\geq 1$, $([b(\frak c+\eta)-a(\frak c-\eta)]f(n)+1)\leq Kf(n)$, which is possible since $f(n)\to+\infty$ with~$n$.

%
Using the law of large numbers in~\eqref{eq:LLN2}, we get that, in probability as $n\to\infty$, 
\[-cQ_n(b)(n-k_n)+\log f(n)
= -c \mathrm e^{(-\mu \hat b/\mathbb EL+o(1))f(n)}(n-k_n)+\log f(n)
\to -\infty,\]
because either $f(n)=o(\log n)$, or $f(n) = \log n$ and $\eta>0$ can be chosen small enough so that $\hat b<\mathbb EL/\mu$ (because, by assumption, $b<\mathbb EL/\mu$, and $\hat b = b(\frak c+\eta)/(\frak c-\eta)$).
This implies that the right-hand side in~\eqref{eq:K} converges to zero as $n\to+\infty$.
Using again that $h(n)\sim f(n)$ in probability when $n\to+\infty$, 
we get that the right-hand side of~\eqref{eq:prob} also tends to zero with~$n$,
and thus \eqref{eq:3rd} is true, which concludes the proof of~\eqref{eq:h_instead_f}.

It only remains to show that~\eqref{eq:h_instead_f} implies~\eqref{eq:fringe_trees};
intuitively, this is true because of~\eqref{eq:cvprob_h/f}. Indeed, we have
\begin{align*}
&\sup_{x\in[a,b]} \Big|\frac{\log (D_n^{\sss (i)}(xf)/n)}{xf(n)}- \frac{\mathbb E[\log \bar Y]}{\mathbb EL}\Big|\\
&\leq \sup_{x\in[a,b]} \Big|\frac{\log (D_n^{\sss (i)}(xf)/n)}{xh(n)\mathbb EL}- \frac{\mathbb E[\log \bar Y]}{\mathbb EL}\Big|
+  \sup_{x\in[a,b]} \Big|\frac{\log (D_n^{\sss (i)}(xf)/n)}{xh(n)\mathbb EL}\Big(1-\frac{h(n)\mathbb EL}{f(n)}\Big)\Big|\\
 &= \frac{\op(1)}{\mathbb EL} + \Big|1-\frac{h(n)\mathbb EL}{f(n)}\Big|  \sup_{x\in[a,b]} \Big|\frac{\log (D_n^{\sss (i)}(xf)/n)}{xh(n)\mathbb EL}\Big|
 = \op\bigg(1+\sup_{x\in[a,b]} \Big|\frac{\log (D_n^{\sss (i)}(xf)/n)}{xh(n)}\Big|\bigg),
\end{align*}
where we have used~\eqref{eq:h_instead_f}, in the first equality, 
and~\eqref{eq:cvprob_h/f} in the second one.
By~\eqref{eq:h_instead_f} and the triangular inequality, 
this last supremum goes to zero in probability when $n\to+\infty$, which concludes the proof.
\end{proof}

We are now ready to prove Theorem~\ref{th:voronoi}.
\begin{proof}[Proof of Theorem~\ref{th:voronoi}]
We let $U_{\sss (1)}(n), \ldots, U_{\sss (k)}(n)$ be the nodes $U_1(n), \ldots, U_k(n)$ ordered in increasing $\bs L$-distance to the root, that is, $|U_{\sss (1)}(n)|_{\bs L}\leq \cdots\leq |U_{\sss (k)}(n)|_{\bs L}$, and let $V_1(n), \ldots, V_k(n)$ be the sizes of their respective Voronoi cells (in that order, i.e.\ the Voronoi cell of $U_{\sss (i)}$ is $V_i(n)$).
We set $\frak m = \mathbb E[\log \bar Y]/\mathbb EL$ and start by showing that, in distribution when $n\to+\infty$,
\begin{equation}\label{eq:cv_other_order}
\Big(\frac{\log (V_2(n)/n)}{v_n}, \ldots, \frac{\log (V_k(n)/n)}{v_n}\Big)
\Rightarrow \frac{\frak m}2 (\Psi_{(2)}-\Psi_{(1)}, \ldots, \Psi_{(k)}-\Psi_{(1)}),
\end{equation}
where $\Psi_{(1)}\leq \cdots\leq \Psi_{(k)}$ is the increasing order statistics of $\Psi_1, \ldots, \Psi_k,$ and with 
\begin{equation}\label{eq:def_vn}
v_n = \begin{cases}
\sqrt{(\log n)/\mu^3} &\text{ if }\mathrm{Var}(L)<+\infty,\\
((\log n)/\mu)^{\nicefrac1\alpha}&\text{ otherwise}.
\end{cases}
\end{equation}
Note that, because $\frak m = -\mathbb E[\log(\nicefrac1{\mathbb E\bar Y})]/\mathbb EL = -\mu/\mathbb EL$, this is equivalent to 
\begin{equation}\label{eq:cv_other_order2}
\Big(\frac{\log (V_2(n)/n)}{(\log n)^{\nicefrac1\alpha}}, \ldots, \frac{\log (V_k(n)/n)}{(\log n)^{\nicefrac1\alpha}}\Big)
\Rightarrow \frac{\frak v}{2\mathbb EL} (\Psi_{(1)}-\Psi_{(2)}, \ldots, \Psi_{(1)}-\Psi_{(k)}),
\end{equation}
where we recall that $\alpha :=2$ when $\mathrm{Var}(L)<+\infty$, and where we have set
\[\frak v = \begin{cases}
\mu^{-\nicefrac12} & \text{ if }\mathrm{Var}(L)<+\infty,\\
\mu^{1-\nicefrac1\alpha} &\text{ otherwise.}
\end{cases}\]
We now show that~\eqref{eq:cv_other_order2} implies~\eqref{eq:VorConvDist}: the only difference between the two is that the entries in the left-hand side of~\eqref{eq:cv_other_order2} are ordered in increasing distance of the respective $U_i(n)$'s to the root, while those in the left-hand side of~\eqref{eq:VorConvDist} are ordered in decreasing sizes of the Voronoi cells. 
However, the convergence in \eqref{eq:cv_other_order2} implies in particular that
\[\mathbb P\big(V_1(n)\geq V_2(n)\geq\cdots\geq V_k(n)\big)
= \mathbb P \big(V_i(n) = V_{\sss (i)}(n), \,\forall i \in \{1, \ldots, k\} \big)
\to 1,
\]
where we recall that, by definition, $V_{\sss (i)}(n)$ is the $i$-th largest of the $k$ Voronoi cells. 
We now let $\mathcal C_n$ denote the event that $V_1(n)\geq V_2(n)\geq\cdots\geq V_k(n)$: we have, for all $x_2, \ldots, x_k<0$,
\[
\mathbb P\bigg(\forall i \in \{2, \ldots,  k\} \,  \colon\,  \frac{\log (V_{\sss (i)}(n)/n)}{(\log n)^{\nicefrac1\alpha}}\geq x_i\bigg)
=\mathbb P\bigg(\forall i \in \{2, \ldots,  k\} \,  \colon\,  \frac{\log (V_i(n)/n)}{(\log n)^{\nicefrac1\alpha}}\geq x_i \text{ and }\mathcal C_n\bigg)+o(1)\]
when $n\to+\infty$, because $\mathbb P(\mathcal C_n)\to 1$. Thus,
\eqref{eq:cv_other_order} entails \eqref{eq:VorConvDist}, and due to the above it is sufficient to establish~\eqref{eq:cv_other_order2}.


For this purpose, for each $n\geq 1$ set 
\begin{equation} \label{eq:Kdef}
K_n := \max_{1\leq i<j\leq k}\{|U_i(n)\wedge U_j(n)|_{\bs L}\}.
\end{equation}
On the event 
\begin{equation}\label{eq:def_En}
\mathcal E_n = \bigcap_{1\leq i<j\leq k}\Big\{\big|U_i(n)\wedge U_j(n)\big|_{\bs L}+\big\lceil\big||U_i(n)|_{\bs L} - |U_j(n)|_{\bs L}\big|/2\big\rceil\geq K_n\Big\},
\end{equation}
for all $1\leq i\neq j\leq k$, 
the $\bs L$-distance to the root of the point where the Voronoi cells of $U_i(n)$ and $U_j(n)$ would meet if we ignored all other $k-2$ points would be at least $K_n$. 
Thus, for all $i\neq \ell:=\mathrm{argmin}_{1\leq j\leq k} |U_j(n)|_{\bs L}$, 
the Voronoi cell of $U_i(n)$ meets the Voronoi cell of $U_\ell(n)$ 
at $\bs L$-distance to the root exceeding~$K_n$, 
implying that the Voronoi cell of $U_i(n)$ for all $i\neq \ell$ is the subtree rooted at the ancestor of $U_i(n)$ closest to the root among all ancestors of $U_i(n)$ whose $\bs L$-distance to the root exceeds
\begin{equation} \label{eq:constSpeed}
\big|U_i(n)\wedge U_\ell(n)\big|_{\bs L}+\frac{\big||U_i(n)|_{\bs L} - |U_\ell(n)|_{\bs L}\big|}2.
\end{equation} 
By Lemmas~\ref{lem:LCA},~\ref{lem:real_heights_finite_var} 
and~\ref{lem:real_height_heavy_tails}, we have
$\mathbb P(\mathcal E_n) \to 1$; 
thus, in the rest of the proof, we work on the event $\mathcal E_n$.

For all $1\leq i\neq j\leq k$, we set 
\[X^{\sss (i,j)}_n = \frac{\big|U_i(n)\wedge U_j(n)\big|_{\bs L}+\big||U_i(n)|_{\bs L} - |U_j(n)|_{\bs L}\big|/2}{v_n},\]
where $v_n$ is as in~\eqref{eq:def_vn}.
Note that, by symmetry, the $X_n^{\sss (i,j)}$ all have the same distribution.
Moreover, by Lemmas~\ref{lem:LCA}, ~\ref{lem:real_heights_finite_var} and~\ref{lem:real_height_heavy_tails} (see also for notation),
in distribution when $n\to+\infty$, jointly for all $1\leq i\neq j\leq k$, we have 
\begin{equation}\label{eq:cvX}
X_n^{\sss (i,j)} \Rightarrow |\Psi_i-\Psi_j|/2,
\end{equation}
where we have set, for all $1\leq i\leq k$,
\begin{equation}\label{eq:def_Psi}
\Psi_i := 
\begin{cases}
\Xi_i& \text{ if }\mathrm{Var}(L)<+\infty,\\
\Upsilon_i(\alpha)&\text{ otherwise}.
\end{cases}
\end{equation}
For all $a, b\in (0,\infty)$, we define the event
\[
\mathcal B_n(a,b) := \big\{\forall 1\leq i<j\leq k\colon X_n^{\sss (i,j)}\in [a,b]\big\}.
\]
Note that, due to~\eqref{eq:cvX}, 
\begin{equation} \label{eq:PBn}
\lim_n \mathbb P(\mathcal B_n(a,b)^c) \to 0 \quad \text{ as }a\to0 \text{ and }b\to+\infty. 
\end{equation}
To prove \eqref{eq:cv_other_order}, we start by setting, 
for any permutation $\sigma\in\frak S_k$,
\[\mathcal A_n(\sigma) = \{|U_{\sigma_1}(n)|_{\bs L}\leq |U_{\sigma_2}(n)|_{\bs L}\leq \cdots\leq |U_{\sigma_k}(n)|_{\bs L}\}.\]
For all $\varepsilon>0$ and $\sigma\in\frak S_k$, we have, using the notation introduced before as well as in Proposition~\ref{prop:fringe_trees}, 
\begin{align} \label{eq:distrEst}
\begin{split}
&\mathbb P\big(\exists i \in \{2, \ldots,  k\} \,  \colon\,  
\log (V_i(n)/n)\geq (\frak m+\varepsilon) X_n^{\sss (\sigma_1,\sigma_i)} v_n 
\text{ and }\mathcal B_n(a,b) \cap \mathcal A_n(\sigma)\big)\\
&
=\mathbb P\bigg(\exists i \in \{2, \ldots,  k\} \,  \colon\, 
\log \Big(\frac{D_n^{\sss (\sigma_i)}\big({X^{\sss (\sigma_1, \sigma_i)}_n} v_n \big)}n\Big)
\geq  (\frak m+\varepsilon) X_n^{\sss (\sigma_1,\sigma_i)}v_n
\text{ and }\mathcal B_n(a,b)\cap \mathcal A_n(\sigma)\bigg)\\
&
\leq \sum_{i=2}^k\mathbb P\bigg(\sup_{x\in [a,b]} \frac{\log (D_n^{\sss (i)}(xv_n)/n)}{xv_n}-\frak m\geq \varepsilon\bigg)
\leq (k-1)\mathbb P\bigg(\sup_{x\in [a,b]} \frac{\log (D_n^{\sss (i)}(xv_n)/n)}{xv_n}-\frak m\geq \varepsilon\bigg)\to 0,
\end{split}
\end{align}
when $n\to+\infty$ by Proposition~\ref{prop:fringe_trees}.
Since $\lim_n \mathbb P(\mathcal B_n(a,b)^c) \to 0$ when $a\to0$ and $b\to+\infty$ due to \eqref{eq:PBn}, we get 
\[\mathbb P\bigg(\exists i \in \{2, \ldots,  k\} \,  \colon\,   \frac{\log (V_i(n)/n)}{X_n^{\sss (\sigma_1,\sigma_i)}v_n}\geq \frak m+\varepsilon \text{ and } \mathcal A_n(\sigma)\bigg)\to 0,
\]
when $n\to+\infty$. 
Therefore, for all $\varepsilon\in(0, -\frak m)$, 
for all $x_i<0$, we have, asymptotically when $n\to+\infty$ that
\begin{align*}
&\mathbb P\bigg(\forall i \in \{2, \ldots,  k\} \,  \colon\,  \frac{\log (V_i(n)/n)}{v_n}\geq x_i\bigg)\\
&\hspace{1cm}=(1+o(1))\sum_{\sigma\in\frak S_k} 
\mathbb P\bigg(\forall i \in \{2, \ldots,  k\} \,  \colon\, 
\frac{\log (V_i(n)/n)}{v_n}\geq x_i 
\text{ and }
\frac{\log (V_i(n)/n)}{X_n^{\sss (\sigma_1,\sigma_i)}v_n}\leq \frak m+\varepsilon
\text{ and }\mathcal A_n(\sigma)\bigg)\\
&\hspace{1cm}
\leq (1+o(1))\sum_{\sigma\in\frak S_k} 
\mathbb P\bigg(\forall i \in \{2, \ldots,  k\} \,  \colon\, 
x_i\leq  X_n^{\sss (\sigma_1,\sigma_i)}(\frak m+\varepsilon)
\text{ and }\mathcal A_n(\sigma)\bigg)\\
&\hspace{1cm}= (1+o(1))\sum_{\sigma\in\frak S_k} 
\mathbb P\bigg(\forall i \in \{2, \ldots,  k\} \,  \colon\, 
X_n^{\sss (\sigma_1,\sigma_i)}\leq \frac{x_i}{\frak m+\varepsilon}
\text{ and }\mathcal A_n(\sigma)\bigg),
\end{align*}
where we have used that $\frak m+\varepsilon<0$.
Since this is true for all $\varepsilon\in(0,-\frak m)$, 
we get
\[\mathbb P\bigg(\forall i \in \{2, \ldots,  k\} \,  \colon\,   \frac{\log (V_i(n)/n)}{v_n}\geq x_i\bigg)
\leq (1+o(1))\sum_{\sigma\in\frak S_k} \mathbb P\bigg(\forall i \in \{2, \ldots,  k\} \,  \colon\,  X_n^{\sss (\sigma_1,\sigma_i)}\leq \frac{x_i}{\frak m}
\text{ and }\mathcal A_n(\sigma)\bigg).
\]
By definition of $\mathcal A_n(\sigma)$ and due to \eqref{eq:cvX}, we thus get
\begin{align*}
&\mathbb P\bigg(\forall i \in \{2, \ldots,  k\} \,  \colon\,  \frac{\log (V_i(n)/n)}{v_n}\geq x_i\bigg)\\
&\leq (1+o(1))\sum_{\sigma\in\frak S_k} 
\mathbb P\bigg(\forall i \in \{2, \ldots,  k\} \,  \colon\, 
|\Psi_{\sigma_1}-\Psi_{\sigma_i}|\leq \frac{2x_i}{\frak m}
\text{ and }\Psi_{\sigma_1}\leq \cdots \leq \Psi_{\sigma_k}\bigg)\\
&= (1+o(1))\mathbb P\bigg(\forall i \in \{2, \ldots,  k\} \,  \colon\, 
\Psi_{(i)}-\Psi_{(1)}\leq \frac{2x_i}{\frak m}\bigg),
\end{align*}
which concludes the proof of \eqref{eq:cv_other_order} and thus of \eqref{eq:cv_other_order2}. \end{proof}

\section{Proof of Theorem \ref{th:territories}}\label{sec:territories}
Before proving Theorem \ref{th:territories}, we prove a central limit theorem that extends the weak law of large numbers proved in Proposition~\ref{prop:fringe_trees}:
\begin{lemma}\label{lem:CLT_random_index}
Using the same notation as in Proposition~\ref{prop:fringe_trees},
assume that $(x_n)_{n\geq 1}$ is a sequence of positive random variables converging  in probability as $n\to+\infty$ to a positive constant $x<C(f)$. 
Then, under the assumptions of Theorem~\ref{th:territories} the following hold true.
\begin{itemize}
\item[{\rm\bf (i)}] If $\mathrm{Var}(L)<+\infty$, then, in distribution when $n\to+\infty$,
\[\frac{\log (D_n^{\sss (i)}(x_n f)/n)-xf(n)\mathbb E[\log \bar Y]/\mathbb EL}{\sqrt{xf(n)}}\Rightarrow \mathcal N\bigg(0, \frac{\mathrm{Var}(\log \bar Y)}{\mathbb EL}+\frac{\mathbb E[\log \bar Y]^2\mathrm{Var}(L)}{(\mathbb EL)^2}\bigg).\]

\item[{\rm\bf (ii)}] If $\mathbb P(L\geq x)= \ell(x)x^{-\alpha}$ for $\alpha\in(1,2)$ and $\ell$ slowly varying at infinity, then, in distribution when $n\to+\infty$,
\[\frac{\log (D_n^{\sss (i)}(x_n f)/n)-xf(n)\mathbb E[\log \bar Y]/\mathbb EL}{(xf(n))^{\nicefrac1\alpha}}
\Rightarrow \frac{\mathbb E[\log(\nicefrac1{\bar Y})]}{(\mathbb EL)^{1+\nicefrac1\alpha}}\,\Upsilon(\alpha),\]
where $\Upsilon(\alpha)$ is an $\alpha$-stable distribution.
\end{itemize}
\end{lemma}

\begin{proof}
In this proof, we use the notation introduced in the proof of Proposition~\ref{prop:fringe_trees}.

{\bf (i)} Fix $\varepsilon>0$ such that $x+\varepsilon<C(f)$ and $x-\varepsilon >0.$ By assumption, the probability of the good events $\mathcal G_n := \{x_n\in [x-\varepsilon, x+\varepsilon]\}$
 tends to~1 as $n\to+\infty$.
On this event, we have
\begin{align*}
\left|\log\Big(\frac{D_n^{\sss (i)}(x_n f)}n\Big)-\log Q_n(x_n)\right|
&= \left|\log\Big(\frac{D_n^{\sss (i)}(x_n f)}{(n-k_n(x_n))Q_n(x_n)}\Big)+\log \Big(1-\frac{k_n(x)}{n}\Big)\right|\\
&\leq \left|\log\Big(\frac{D_n^{\sss (i)}(x_n f)}{(n-k_n(x_n))Q_n(x_n)}\Big)\right|
+ \bigg|\log \Big(1-\frac{k_n}{n}\Big)\bigg|,
\end{align*}
because $k_n(x_n) \leq k_n$, 
by definition.
By~\eqref{eq:3rd}, 
we have on $\mathcal G_n$ that
\[\left|\log\Big(\frac{D_n^{\sss (i)}(x_n f)}{(n-k_n(x_n))Q_n(x_n)}\Big)\right|
\leq \sup_{x_n\in[x-\varepsilon, x+\varepsilon]} \left|\log\Big(\frac{D_n^{\sss (i)}(x f)}{(n-k_n(x))Q_n(x)}\Big)\right|
= \op(1).\]
We thus get that on $\mathcal G_n,$
\[\left|\log\Big(\frac{D_n^{\sss (i)}(x_n f)}n\Big)-\log Q_n(x_n)\right|
= \log(1-\widetilde U) + \op(1),\]
which implies that
\begin{equation} \label{eq:O1Approx}
\left|\log\Big(\frac{D_n^{\sss (i)}(x_n f)}n\Big)-\log Q_n(x_n)\right|
=\Op(1)\quad\text{ as }n\to+\infty.
\end{equation}

Now recall that
\[\log  Q_n(x_n) = \sum_{\varnothing\neq u\preccurlyeq \xi(k_n(x_n))} \log Z_u\]
is a sum of $x_n h(n)$ i.i.d.\ random variables with finite variance,
since $\mathrm{Var}(\log \bar Y) <+\infty$ by assumption.
By definition, for all $y>0$, the node $\xi(k_n(y))$ is the ancestor of $U_i(n)$ closest to the root
whose $\bs L$-distance to the root is at least $yh(n)$. 
Therefore, almost surely, for all $y\leq z$, 
$\xi(k_n(y))$ is an ancestor of $\xi(k_n(z))$ 
(which includes the case when both nodes are equal).
In other words, as $y$ increases from $x-\varepsilon$ to $x+\varepsilon$, 
$\xi(k_n(y))$ goes through the ancestors of $U_i(n)$ at $\bs L$-distance to the root between $(x-\varepsilon)h(n)$ and $(x+\varepsilon)h(n)$, in that order. 
Therefore, in distribution, we have, jointly for all $y\in[x-\varepsilon,x+\varepsilon]$, 
\[\log  Q_n(y) 
= \sum_{i=1}^{|\xi(k_n(y))|}\log \widetilde Z_i = \sum_{i=1}^{yh(n)}\log \widetilde Z_i,\]
where the sequence $(\widetilde Z_i)_{i\geq 1}$ is a sequence of i.i.d.\ random copies of~$\bar Y$, 
because, by definition, $|\xi(k_n(y))| = yh(n)$ for all $y>0$.
Thus, by the central limit theorem with random index (see, e.g., \cite[Exercise~3.4.6]{Durrett}),
we get
\begin{equation}\label{eq:first_CLT}
\frac{\log  Q_n(x_n) - xh(n)\mathbb E[\log\bar Y]}{\sqrt{xh(n)}}
\Rightarrow \mathcal N(0,\mathrm{Var}(\log \bar Y)).
\end{equation}
Also recall that, by definition, $xF_i(n) = \sum_{i=1}^{xh(n)} L_i$ in distribution, and thus applying the central limit theorem, but this time to the sequence $\bs L$, we get that, if $\mathrm{Var}(L)<+\infty$, then 
\begin{equation}\label{eq:sec_CLT}
\frac{xF_i(n)-xh(n)\mathbb EL}{\sqrt{xh(n)}} \Rightarrow \mathcal N(0,\mathrm{Var}(L)).
\end{equation}
Note that, by the independence of the sequence $\bs L$ and the rest of the process, the two limits in~\eqref{eq:first_CLT} and~\eqref{eq:sec_CLT} hold jointly, and the two Gaussians are independent.
Combining \eqref{eq:O1Approx} to \eqref{eq:sec_CLT}, the above yields
\begin{align*}
&\sqrt{xF_i(n)}\left(\frac{\log (D_n^{\sss (i)}(x_n f)/n)}{xF_i(n)} - \frac{\mathbb E[\log \bar Y]}{\mathbb EL}\right)\\
&\hspace{2cm}=
\sqrt{\frac{h(n)}{F_i(n)}} \left(\frac{\log  Q_n(x_n) +\Op(1)-xh(n)\mathbb E[\log\bar Y]}{\sqrt{xh(n)}} 
+\frac{\mathbb E[\log \bar Y]}{\mathbb EL}\cdot \frac{xh(n)\mathbb EL - x_n F_i(n)}{\sqrt{xh(n)}}\right)\\
&\hspace{2cm}\Rightarrow \mathcal N\bigg(0,\frac{\mathrm{Var}(\log \bar Y)}{\mathbb EL}+\frac{\mathbb E[\log \bar Y]^2\mathrm{Var}(L)}{(\mathbb EL)^2}\bigg).
\end{align*}
In~\eqref{eq:cvprob_Fi/f} we have proved that $F_i(n)\sim f(n)$ in probability as $n\to+\infty$, but in fact, this statement can be made stronger: we have $F_i(n) = f(n)$ as soon as $|U_i(n)|_{\bs L}\geq xf(n)$, an event whose probability tends to~1 with~$n$ because either $f(n)= o(\log n)$ or $f(n) = \log n$ and $x<\mathbb EL/\mu$.
This implies~(i).

{\bf (ii)} Under the assumption that $\mathbb P(L\geq x)=\ell(x)x^{-\alpha}$ with $\alpha\in (1, 2)$,
the limit in \eqref{eq:sec_CLT} does not hold, instead we have that
\[\frac{xF_i(n)-xh(n)\mathbb EL}{(xh(n))^{\nicefrac1\alpha}}
\Rightarrow \Upsilon(\alpha),\]
where $\Upsilon(\alpha)$ is an $\alpha$-stable distribution.
Thus
\begin{align*}
&\frac{\log (D_n^{\sss (i)}(x_nf)/n)-xF_i(n)\mathbb E[\log \bar Y]/\mathbb EL}{(xF_i(n))^{\nicefrac1\alpha}}\\
&=\bigg(\frac{h(n)}{F_i(n)}\bigg)^{\!\nicefrac1\alpha}
\left(\frac{\mathbb E[\log \bar Y]}{\mathbb EL}\cdot \frac{xh(n)\mathbb EL - xF_i(n)}{(xh(n))^{\nicefrac1\alpha}}
+ \frac{\log  Q_n(x) +\Op(1) - xh(n)\mathbb E[\log\bar Y]}{(xh(n))^{\nicefrac1\alpha}}\right)
\Rightarrow \frac{\mathbb E[\log(\nicefrac1{\bar Y})]}{(\mathbb EL)^{1+\nicefrac1\alpha}}\Upsilon(\alpha),
\end{align*}
where the second summand now vanishes in the limit.
This concludes the proof of (ii) because $\mu = \mathbb E[\log(\nicefrac1{\bar Y})]$, and because $F_i(n) = f(n)$ with probability tending to~1 when~$n$ tends to infinity.
\end{proof}

\begin{proof}[Proof of Theorem \ref{th:territories}]
For this proof, we consider that the infections ``creep along edges'' between the times at which they infect vertices: if two infections of respective speeds $s$ and $s'$ start at two neighbouring vertices $v$ and $v'$ (respectively), then they meet at distance from $v$ proportional to $\frac s{s+s'}$.

For any two infections $i$ and $\ell$, 
the $\bs L$-distance between $U_i(n)$ and $U_\ell(n)$ is equal to
\[\Delta_{i,\ell} := |U_i(n)|_{\bs L}+|U_\ell(n)|_{\bs L}-2|U_i(n)\wedge U_\ell(n)|_{\bs L}.\]
Therefore, the time $t_{i,\ell}$ at which epidemics $i$ and $\ell$ would meet if there were no other infection at play is equal to the time it would take for an infection of speed $s_i+s_\ell$ to cross a distance $\Delta_{i,\ell}$, i.e.
\begin{equation}\label{eq:til}
t_{i,\ell} = \frac{\Delta_{i,\ell}}{s_i+s_\ell}
= \frac{|U_i(n)|_{\bs L}+|U_\ell(n)|_{\bs L}-2|U_i(n)\wedge U_\ell(n)|_{\bs L}}{s_i+s_\ell}.
\end{equation}
Therefore, in the absence of the other $k-2$ infections, the $i$-th and $\ell$-th infections would meet 
at $\bs L$-distance to the root equal to the maximum of
\begin{equation}\label{eq:meeting}
|U_i(n)|_{\bs L} - s_i t_{i,\ell} \quad\text{ and }\quad
|U_\ell(n)|_{\bs L} - s_\ell t_{i,\ell}.
\end{equation}
Thus, on the event (recall $K_n$ from \eqref{eq:Kdef}),
\[\mathcal A_n 
= \bigcap_{1\leq i<\ell\leq k}\Big\{|U_\ell(n)|_{\bs L} - s_i t_{i,\ell}\geq K_n\Big\},\]
for all $1\leq i<\ell\leq k$, 
if we ignored all other $k-2$ epidemics,
the epidemics started respectively at $U_i(n)$ and $U_\ell(n)$ would meet at $\bs L$-distance to the root at least $K_n$.
Also note that the probability of $\mathcal A_n$ goes to one when $n\to+\infty$ by Lemmas~\ref{lem:LCA} and~\ref{lem:real_heights_finite_var}. 
Thus it is enough to restrict ourselves to the set where $\mathcal A_n$ holds.

We let $\kappa = \kappa(n) = \argmin \{1\leq \ell\leq j\colon |U_\ell(n)|_{\bs L}\}$. 
On the event $\mathcal A_n$, for all $i\neq \kappa$, 
the territory of the $i$-th infection neighbours a unique other territory, 
and this neighbouring territory is the territory of the $\kappa$-th epidemic.
We let $d_i(n)$ denote the $\bs L$-distance 
from the root to the point where they meet (this point can be in the middle of an edge). 
On $\mathcal A_n$, the territory of the $i$-th infection is the subtree of $\tau_n$ rooted at
the ancestor of $U_i(n)$ closest to the root whose $\bs L$-distance to the root is at least $d_i(n)$.
We now show that
\begin{equation} \label{eq:miConv}
\frac{d_i(n)}{\log n} \to 
\frac{s_1-s_i}{s_1+s_i}\cdot\frac{\mathbb EL}{\mu}\quad\text{ in probability as }n\to+\infty.
\end{equation}
Indeed,
first note that, by~\eqref{eq:til} and~\eqref{eq:meeting},
under $\mathcal A_n$, 
the $i$-th and $\kappa$-th infections meet at $\bs L$-distance to the root equal to
$|U_i(n)|_{\bs L} - s_i t_{i, \kappa}$, and thus
\[d_i(n) = |U_i(n)|_{\bs L}-s_i t_{i,\kappa}.\]
By Lemma~\ref{lem:LCA}, and using the notation $H_n=\max_{1\le i < k \le n} |U_i(n)\wedge U_k(n)|$, we get
\begin{equation}\label{eq:top_height_negl}
0\leq \frac{|U_i(n)\wedge U_\kappa(n)|_{\bs L}}{\log n}
\leq \frac{H_n}{\log n} \to 0 \quad\text{ in probability as }n\to+\infty,
\end{equation}
implying that, as $n\to+\infty$,
\begin{equation}\label{eq:approx_tcross}
t_{i, \kappa} = \frac{|U_i(n)|_{\bs L}+|U_\kappa(n)|_{\bs L}}{s_1+s_i} + \op(\log n),
\end{equation}
which yields
\[d_i(n) 
= \frac{s_1|U_i(n)|_{\bs L}-s_i |U_\kappa(n)|_{\bs L}}{s_1+s_i} + \op(\log n),\quad\text{ as }n\to+\infty.\]
Since, by Lemmas~\ref{lem:real_heights_finite_var} and~\ref{lem:real_height_heavy_tails}, $U_i(n)=(\mathbb EL/\mu+\op(1)) \log n$ as $n\to+\infty$ (and similarly for $U_\kappa(n)$), we get~\eqref{eq:miConv}.

As argued above, on $\mathcal A_n$, $\Ter_{t,d}^{\sss (i)}((U_1, s_1), \ldots, (U_k, s_k))$ is the subtree of $\tau_n$ rooted at the ancestor of $U_i(n)$ closest to the root whose $\bs L$-distance is at least $x^{\sss (i)}_n\log n$, where, by~\eqref{eq:miConv},
\[x^{\sss (i)}_n := \frac{d_i(n)}{\log n}
=\frac{s_1-s_i}{s_1 + s_i}\cdot\frac{\mathbb E L}{\mu} + \op(1).\]
We set $x^{\sss (i)} = \lim_{n\to+\infty} x_n^{\sss (i)}$ (with the limit holding in probability).
Thus, by Lemma~\ref{lem:CLT_random_index}(i), and because $W_{(i)}(n) = D_n^{\sss (i)}(x^{\sss (i)}_n f)$, we get
\[\sqrt{x^{\sss (i)}\log n}\left(\frac{\log(W_{(i)}(n)/n)}{x^{\sss (i)}\log n}-\frac{\mathbb E[\log \bar Y]}{\mathbb EL}\right)\Rightarrow
\mathcal N\bigg(0, \frac{\mathrm{Var}(\log \bar Y)}{\mathbb EL}+\frac{\mathbb E[\log \bar Y]^2\mathrm{Var}(L)}{(\mathbb E L)^2}\bigg),\]
if the random variable $L$ has finite variance. 
This implies
\[\frac{\log(W_{(i)}(n)/n) + \frac{s_1-s_i}{s_1+s_i}\log n}{\sqrt{\frac{s_1-s_i}{s_1+s_i}\log n}}
\Rightarrow \mathcal N\bigg(0, \frac{\mathrm{Var}(\log \bar Y)\mathbb EL}{\mathbb E[\log\bar Y]^2}+\mathrm{Var}(L)\bigg),
\]
as claimed.
In the case when the edge length are heavy-tailed, we get from Lemma~\ref{lem:CLT_random_index}(ii) that
\[\frac{\log (W_{(i)}(n)/n)-x^{\sss (i)}(\log n)\mathbb E[\log \bar Y]/\mathbb EL}{(x^{\sss (i)}\log n)^{\nicefrac1\alpha}}
\Rightarrow \frac{\mu}{(\mathbb E L)^{1+\nicefrac1\alpha}}\Upsilon(\alpha),\]
which implies
\[\frac{\log (W_{(i)}(n)/n)+\frac{s_1-s_i}{s_1+s_i} \log n}{(\frac{s_1-s_i}{s_1+s_i} \log n)^{\nicefrac1\alpha}}
\Rightarrow \frac{\mu^{1-\nicefrac1\alpha}}{\mathbb E L}\Upsilon(\alpha).
\]
Because the probability of $\mathcal A_n$ converges to~1 as $n\to+\infty$, this concludes the proof of~\eqref{eq:terMinorConv}.
%

To prove the convergence in \eqref{eq:terPrincipalConv}, note that
 it follows from Theorem \ref{th:voronoi} and  \eqref{eq:terMinorConv} that the sizes $W_{(i)}(n)$, $j+1 \le i \le k$ are $o(W_{(\ell)}(n))$ for all $\ell \in \{1, \ldots, j\}$ 
with high probability. This implies  \eqref{eq:terPrincipalConv}.
\end{proof}

\textbf{Acknowledgement.} We acknowledge support from DFG through the scientific network
\emph{Stochastic Processes on Evolving Networks}.

\bibliographystyle{alpha}
\bibliography{Voronoi}

\newcommand{\etalchar}[1]{$^{#1}$}
\begin{thebibliography}{ADMP17}

\bibitem[AAC{\etalchar{+}}18]{Chapuy}
Louigi {Addario-Berry}, Omer {Angel}, Guillaume {Chapuy}, \'Eric {Fusy}, and
  Christina {Goldschmidt}.
\newblock {Voronoi tessellations in the CRT and continuum random maps of finite
  excess.}
\newblock In {\em {Proceedings of the 29th annual ACM-SIAM symposium on
  discrete algorithms, SODA 2018, New Orleans, LA, USA, January 7--10, 2018}},
  pages 933--946. Philadelphia, PA: Society for Industrial and Applied
  Mathematics (SIAM); New York, NY: Association for Computing Machinery (ACM),
  2018.

\bibitem[ADMP17]{Antun17}
Ton\'ci {Antunovi\'c}, Yael {Dekel}, Elchanan {Mossel}, and Yuval {Peres}.
\newblock {Competing first passage percolation on random regular graphs.}
\newblock {\em {Random Struct. Algorithms}}, 50(4):534--583, 2017.

\bibitem[BHK15]{BaronHofstKomja15}
Enrico {Baroni}, Remco van~der {Hofstad}, and J\'ulia {Komj\'athy}.
\newblock {Fixed speed competition on the configuration model with infinite
  variance degrees: unequal speeds}.
\newblock {\em {Electron. J. Probab.}}, 20:48, 2015.
\newblock Id/No 116.

\bibitem[CDJH01]{CDJ01}
B.~Chauvin, M.~Drmota, and J.~Jabbour-Hattab.
\newblock The profile of binary search trees.
\newblock {\em Annals of Applied Probability}, pages 1042--1062, 2001.

\bibitem[{Cha}19]{Chapu19}
Guillaume {Chapuy}.
\newblock {On tessellations of random maps and the \(t_g\)-recurrence.}
\newblock {\em {Probab. Theory Relat. Fields}}, 174(1-2):477--500, 2019.

\bibitem[Dev98]{devroye-split}
Luc Devroye.
\newblock Universal limit laws for depths in random trees.
\newblock {\em SIAM Journal on Computing}, 28(2):409--432, 1998.

\bibitem[DG97]{DG}
Michael Drmota and Bernhard Gittenberger.
\newblock On the profile of random trees.
\newblock {\em Random Structures \& Algorithms}, 10(4):421--451, 1997.

\bibitem[DH16]{deijfen2016}
Maria Deijfen and Remco van~der Hofstad.
\newblock The winner takes it all.
\newblock {\em Ann. Appl. Probab.}, 26(4):2419--2453, 08 2016.

\bibitem[Dur19]{Durrett}
Rick Durrett.
\newblock {\em Probability---theory and examples}, volume~49 of {\em Cambridge
  Series in Statistical and Probabilistic Mathematics}.
\newblock Cambridge University Press, Cambridge, 2019.
\newblock Fifth edition.

\bibitem[GK54]{GK}
B.V.\ Gnedenko and A.N.\ Kolmogorov.
\newblock {\em Limit distributions for sums of independent random variables}.
\newblock Addison-Wesley, 1954.

\bibitem[Gui17]{Guitter_2017}
Emmanuel Guitter.
\newblock On a conjecture by chapuy about voronoi cells in large maps.
\newblock {\em Journal of Statistical Mechanics: Theory and Experiment},
  2017(10):103401, oct 2017.

\bibitem[HJ17]{JansonHolm}
Cecilia Holmgren and Svante Janson.
\newblock Fringe trees, {C}rump--{M}ode--{J}agers branching processes and
  $m$-ary search trees.
\newblock {\em Probability Surveys}, 14:53--154, 2017.

\bibitem[HK15]{HofstKomja15}
Remco van~der {Hofstad} and J\'ulia {Komj\'athy}.
\newblock Fixed speed competition on the configuration model with infinite
  variance degrees: equal speeds, 2015.
\newblock Eprint ar{X}iv:1503.09046 [math.{PR}].

\bibitem[Jan19]{Janson}
Svante Janson.
\newblock Random recursive trees and preferential attachment trees are random
  split trees.
\newblock {\em Combinatorics, Probability and Computing}, 28(1):81--99, 2019.

\bibitem[Kat05]{Katona05}
Z.~Katona.
\newblock Width of a scale-free tree.
\newblock {\em J. Appl. Probab.}, 42(3):839--850, 2005.

\bibitem[MM17]{MM17}
C.~Mailler and J.-F. Marckert.
\newblock Measure-valued {P}{\'o}lya urn processes.
\newblock {\em Electronic Journal of Probability}, 22, 2017.

\end{thebibliography}

\end{document}